\documentclass[oneeqnum,onefignum,onetabnum,twoside]{siamart190516}
\usepackage{amsfonts}
\usepackage{graphicx}
\usepackage{epstopdf}
\usepackage{algorithmic}
\ifpdf
  \DeclareGraphicsExtensions{.eps,.pdf,.png,.jpg}
\else
  \DeclareGraphicsExtensions{.eps}
\fi

\usepackage{amsopn}
\DeclareMathOperator{\diag}{diag}

\usepackage{tikz}
\usetikzlibrary{arrows.meta}
\usepackage{tkz-graph}
\usetikzlibrary{arrows}

\usepackage{dsfont}

\usepackage{amsmath,amssymb}

\newsiamremark{rmrk}{Remark}
\newsiamthm{assumption}{Assumption}
\newsiamremark{xmpl}{Example}

\newcommand{\xR}{\mathbb{R}}
\newcommand{\xN}{\mathbb{N}}
\newcommand{\xC}{\mathbb{C}}

\newcommand{\vertiii}[1]{{\left\vert\kern-0.25ex\left\vert\kern-0.25ex\left\vert #1 
        \right\vert\kern-0.25ex\right\vert\kern-0.25ex\right\vert}}

\newcommand{\InitCond}[1]{#1^{0}} 
\newcommand{\uuuu}{u}

\newcommand{\newN}{M}

\headers{Stability for time-varying delay systems or Hyperbolic PDEs}{L. Baratchart, S. Fueyo, G. Lebeau, J.-B. Pomet}

\title{Sufficient stability conditions for time-varying networks of
  telegrapher's equations or \\difference-delay
  equations\thanks{
    November 22, 2019, revised July 21, 2020 and January 18, 2021.
    \\
    To appear in \textit{SIAM J.\ on Math.\ Analysis.}
    \hfill
    \textcopyright\ 2021 Society for Industrial and Applied Mathematics. 
  }}

\author{L. Baratchart\thanks{Inria, Université Côte d'Azur, Teams FACTAS and MCTAO, 2004, route des
  Lucioles, 06902 Sophia
  Antipolis, France (\email{Laurent.Baratchart@inria.fr}, \email{Sebastien.Fueyo@inria.fr}, \email{Jean-Baptiste.Pomet@inria.fr}).}
\and S. Fueyo\footnotemark[2]
\and G. Lebeau\thanks{Laboratoire J.A. Dieudonné, UMR CNRS 7351,
  Université C\^ote d'Azur, Parc Valrose, 06108 Nice, France (\email{Gilles.Lebeau@unice.fr}).}
\and J.-B. Pomet\footnotemark[2]
}

\ifpdf
\hypersetup{
  pdftitle={Sufficient stability conditions for time-varying networks of telegrapher's equations or difference-delay equations},
  pdfauthor={L. Baratchart, S. Fueyo, G. Lebeau and J.-B. Pomet}
}
\fi

\begin{document}

\maketitle

\begin{abstract}
  We give a sufficient condition for exponential stability
  of a network of lossless telegrapher's equations,
  coupled by linear time-varying boundary conditions.
  The sufficient conditions is in terms of dissipativity of the couplings,
  which is natural for instance in the context of  microwave circuits.
  Exponential stability is with respect to any $L^p$-norm, $1\leq p\leq\infty$.
  This also yields a sufficient condition for exponential stability to
  a special class of systems of linear
  time-varying difference-delay equations which is quite explicit and tractable.
  One ingredient of the proof is that $L^p$ exponential stability for such difference-delay systems
  is independent of $p$, thereby reproving in a simpler way some
  results from \cite{Chitour2016}.
\end{abstract}



\section{Introduction}
\label{sec:intro}
The stability of electrical circuits operating at high frequency, that is,
when delays induced by wires cannot be neglected, 
has received a lot of attention in the last decades, see for example 
references \cite{Brayton_reseau,Rasvan}. At such an operating regime, wires 
should be considered as transmission lines, and
it is customary to model each of them  by a lossless telegrapher's equation
(a 1-D hyperbolic partial differential equation, in short: hyperbolic PDE)
where voltage and current are functions of abscissa and time.
The other elements in the circuit, some of which may be active and nonlinear (transistors, diodes), induce couplings between the boundary conditions of
these PDE
consisting of a system of  both differential and non-differential equations with finite-dimensional state,
obtained by applying the classical laws of electricity,
  at each node, to the boundaries that ``touch'' this node.

Periodic solutions for such infinite dimensional dynamical systems occur 
naturally in several contexts; for instance, they arise spontaneously in the case
of oscillators, or through periodic forcing in the case of amplifiers 
(the forcing is the
signal to be amplified, represented for instance by a periodic voltage source).
Assuming a periodic solution, one may linearize the equations around the latter to investigate its local exponential stability, 
based on the exponential stability of the first order approximation.
The linearized system consists of the original collection of telegrapher's equations  (which are linear already),
coupled at their nodes ({\it i.e.} the endpoints of a line)
by a set of linear differential and non-differential equations with periodic 
coefficients,
obtained by linearizing the initial couplings, see \cite{Suarez}.
To this linear system, one associates a \emph{high frequency limit system}
(in short: HFLS), where the linear differential equation at each node 
degenerates into a linear, time-varying but non-differential 
relation ({\it i.e.} there is no dynamics in the couplings at infinite frequencies), so that the state of the 
HFLS reduces to
currents and voltages in the lines.
The behavior of the HFLS  is crucial to 
the stability of the linearized system,
because the solution operator of the latter is, in natural functional spaces, 
a compact
perturbation of the solution operator to the HFLS, see
\cite[ch. 3, thm. 7.3]{Hale} and \cite{SFueyo}.
In particular, the stability of the HFLS is essentially necessary to the stability
of the linearized system.

The HFLS is a system of lossless 1-D telegrapher's 
equations, with
linear couplings that depend on time in a periodic manner.
With this application in mind, the present paper is devoted, more generally,
to the stability of lossless 1-D telegrapher's equations with
linear time-varying couplings whose coefficients are
measurable and uniformly essentially bounded with respect to time, but not
necessarily periodic.
As is well known,  integrating
the telegrapher's equation yields an expression of the general solution in
terms of two functions  of one variable, 
and this allows one to recast the original system as a system of
time-varying linear \emph{difference-delay equations}; the two frameworks 
are equivalent to study issues of stability.

\smallskip

Stability of networks of hyperbolic PDEs has been addressed extensively, 
including more
general systems of conservation laws than telegrapher's equations (possibly nonlinear), 
but almost\footnote{\label{foot-1}
  As an anonymous reviewer
  pointed out to us, although the paper \cite{CoNg} deals with local
  stability of an equilibrium point for nonlinear \emph{time-invariant} hyperbolic systems, it contains a statement (Lemma 3.2) about
  stability of smoothly \emph{time-varying} linear systems of
  hyperbolic PDEs for some Sobolev norm. We discuss this further in
  Section~\ref{sec:results-nous} and sketch in
  Section~\ref{sec:proof-gen_TDS} how the proof of that lemma may be
  adapted here.
  }
only when the boundary conditions 
({\it i.e.}\ the couplings)  consist of \emph{time-independent} relations, 
see \cite{bastin2016stability,CoNg} and the bibliography therein.
Another possible, different application of these
  criteria is to
  stabilization of such equations with control, like in \cite{CVKB,hale2002strong} for instance.
As far as methods are concerned, Lyapunov functions are a classical tool to obtain
  sufficient stability conditions, see
\cite{bastin2016stability} where they are applied to certain systems of
hyperbolic PDEs with conservation laws, or for instance
\cite{Fridman2009}, where Lyapunov functions are
constructed through linear matrix inequalities, to retarded delay systems.
We are not aware of attempts in this direction for difference-delay systems.

In another connection, a typical way of obtaining
necessary and sufficient stability conditions for a time-invariant
network of telegrapher's equations is to apply the Henry-Hale theorem~\cite{Henry1974,Hale} or 
variants thereof ({\it cf.}\ Section \ref{sec:results}) to
the equivalent difference-delay system with constant coefficients.
However,  no analog
for the time-dependent case seems to be known.

The main contribution of this paper is to establish sufficient conditions for exponential stability
of networks of telegrapher's equations, in the form of
a  dissipativity assumption on the couplings at each node of the network, which is fairly natural in
a circuit-theoretic context.
We also derive sufficient conditions for exponential stability of time-varying difference-delay
systems, that are a consequence of the former and of independent interest.
To our knowledge, this is the first result of this kind in the time-varying case.
The proof, which involves  going back and forth
between the PDE formulation and the difference-delay system formulation, has interesting features that  should
be useful in other contexts as well. 
Roughly speaking, we rely on classical energy estimates
to first obtain  a Lyapunov function in the $L^2$ sense for each 
telegrapher's equation,
using the dissipativity condition at each node; this  allows us to
show $L^2$ 
exponential stability of the system of PDE, therefore also of the 
associated delay  system. 
In a second step, we deduce from  
the $L^2$ exponential stability of the difference-delay system
its exponential stability in the $L^\infty$ sense 
(and in fact in the $L^p$ sense for all $p\in[1,\infty]$). 
This second step is actually 
subsumed under the work in \cite{Chitour2016}, but we feel our derivation is 
simpler and worthy in its own. Note that applications to the
local stability of a periodic trajectory in an electrical network indeed 
require $L^\infty$ (or $C^0$) stability and not just $L^2$ 
stability, for the state along a perturbed trajectory of the linearized 
system must remain uniformly close to the state along the periodic 
trajectory of the original system, in order that  linearization remains 
meaningful. 
  This paper makes no attempt at handling more general PDEs or coupling conditions. We rather tried to remain as elementary as
possible in treating the problem at hand.
In particular, our arguments for well-posedness may fail for general
hyperbolic 1-D equations, for which notions like broad
solutions were introduced in \cite{Bres00} and used \textit{e.g.} in
\cite{Coro-Ngu19}, see also \cite{bastin2016stability} for other approaches.

\smallskip

The paper is organized as follows. Section \ref{sec:pb} 
introduces networks of 
telegrapher's equations coupled by time-varying boundary conditions, 
gives well-posedness results that  we could not find in the literature,
discusses the construction of equivalent difference-delay systems and defines the
notions of  stability under examination. Section~\ref{sec:results} 
contains our main result, both  in terms of 
networks of telegrapher's equations
and in terms of difference-delay equations. Section~\ref{sec:proofs} is devoted to the proofs.

\section{Problem statement}
\label{sec:pb}

\subsection{A time-varying network of hyperbolic equations}
\label{sec:network}

Consider a directed graph with $N$ edges and $N'$ nodes, where $N$ and $N'$ are two positive integers.
Nodes are numbered by integers $p\in\{1,\cdots,N'\}$,
and edges by integers $k\in\{1,\cdots,N\}$.

Figure~\ref{fig:1} represents a graph with 3 nodes and 4 edges whose only purpose is to illustrate the definitions.

\begin{figure}[h]\label{fig:1}
\begin{center}
    \begin{tikzpicture}
      \begin{scope}[every node/.style={circle,thick,draw}]
        \node (A) at (0,3) {1}; \node (B) at (2.5,3) {2}; \node (C) at (5,3) {3} ;
      \end{scope}
      \begin{scope}[>={Stealth[black]}, every node/.style={fill=white,circle}, every
        edge/.style={draw=black,very thick}]
        \path [->] (A) edge node {$1$} (B); \path [->] (B) edge[bend right=30] node {$2$} (C); \path
         [->] (C) edge[bend right=30] node {$3$} (B); \path [->] (B) edge node {$4$} (C);
       \end{scope}
     \end{tikzpicture}
  \caption{A graph that induces coupling boundary conditions for \eqref{eq_tel} with $N=4$.}
\end{center}
\end{figure}
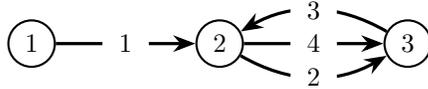

\textbf{Each edge} figures a telegrapher equation. More precisely, we see edge number $k$ as a copy of the real
segment $[0,1]$ ({\it i.e.} a transmission line of unit length) on which
two real functions $v_k(t,.)$ and $i_k(t,.)$ (the voltage and the
current) are defined for any positive time $t$
in such a way that the lossless telegrapher's equation is satisfied:
\begin{eqnarray}
\label{eq_tel} 
\left\{
  \begin{array}{ll}
    \displaystyle
    C_k \frac {\partial v_k(t,x)}{\partial t}=-\frac {\partial i_k(t,x)}{\partial x}, & \hspace{1cm} \\[2ex]
    \displaystyle
    L_k \frac {\partial i_k(t,x)}{\partial t}=-\frac {\partial v_k(t,x)}{\partial x},& \hspace{1cm}
  \end{array}
\right.
(t,x)\in \Omega\,, 
\end{eqnarray}
where 
\begin{equation} 
  \label{eq:Omega}
    \Omega\;=\;\{(t,x)\in\xR^2,\;0<x<1\text{ and }0<t<+\infty\},
\end{equation}
and $L_k,C_k$ are strictly  positive numbers called the inductance 
and capacity of the line, from which
the delay $\tau_k$ and characteristic impedance $K_k$ may be
defined by \eqref{eq:1}. Possibly re-ordering the edges, we assume
that the delays $\tau_k$ are increasing:
\begin{equation}
  \label{eq:1}
  \tau_k=\sqrt{L_k C_k}\,,\ \ K_k=\sqrt{C_k/L_k}\,,\ \
  0<\tau_1\leq \tau_2\cdots \leq \tau_N\,.
\end{equation}

\textbf{Each node} couples the edges adjacent to it through boundary 
conditions at one of the endpoints of $[0,1]$.
More precisely, for node number $p$, if
\begin{itemize}
\item $j(p)$ is the out-degree of the graph at this node, \textit{i.e.} the number of edges outgoing from
  it, these edges being labelled $k_{1}<\cdots<k_{j(p)}$,
\item $\tilde{\jmath}(p)$ is the in-degree of the graph at this node, \textit{i.e.} the number of edges
  incoming at it, these edges being labelled $k_{j(p) +1}<\cdots<k_{j(p) +\tilde{\jmath}(p)}$,
\end{itemize}
the node couples the equations \eqref{eq_tel} corresponding to these
instances of $k$ \textit{via} the following $j(p)+\tilde{\jmath}(p)$ linear relations:
\begin{eqnarray}
\label{equ_Boundary}
V_p(t)=A_p(t) I_p(t), 
\end{eqnarray}
on the $2(j(p)+\tilde{\jmath}(p))$ entries of the following two vectors:
\begin{equation}
  \label{eq:defVI}
  V_p(t)=
  \begin{pmatrix}
    v_{k_1}(t,0) \\
    \vdots \\
    v_{k_{j(p)}}(t,0) \\
    v_{k_{j(p)+1}}(t,1)\\
    \vdots \\
    v_{k_{j(p)+\tilde{\jmath}(p)}}(t,1)
  \end{pmatrix}
  \,,\quad\quad 
  I_p(t)=
  \begin{pmatrix}
    -i_{k_1}(t,0) \\
    \vdots \\
    -i_{k_{j(p)}}(t,0) \\
    i_{k_{j(p)+1}}(t,1)\\
    \vdots \\
    i_{k_{j(p)+\tilde{\jmath}(p)}}(t,1) 
  \end{pmatrix}
  \,.
\end{equation}
In \eqref{equ_Boundary}, $t\mapsto A_p(t)$ is a map from $[0,+\infty)$ 
to the set of square
$(j(p)+\tilde{\jmath}(p))\times (j(p)+\tilde{\jmath}(p))$ matrices, that
will always be assumed measurable and bounded. It is
moreover continuous in many cases of interest ({\it e.g.} when modeling an electrical circuit).
Our results rest on the following condition characterizing 
\emph{dissipativity} at
(each) node $p$:
\begin{eqnarray}\label{eq:dissip-p}
A_p(t)+A_p^{*}(t) \geq \alpha_p\,Id,\ \ \alpha_p>0\,,\ \ p\in\{1,\cdots,N'\},
\end{eqnarray}
where $\alpha_p$ is independent of $t$
and superscript $*$ denotes the transpose of a real matrix, or
the transpose conjugate of a complex matrix since we shall have an occasion to deal also with complex 
matrices.
Inequality \eqref{eq:dissip-p} is meant to hold between symmetric
matrices, for a.e. $t$.
Here and below, the symbol $Id$ stands for the identity operator or 
matrix of appropriate size, 
while the context will keep the meaning clear.
\begin{xmpl}
  For the graph in Figure~\ref{fig:1} it holds that $N=4$, i.e. we have four 
telegrapher's equations of the form \eqref{eq_tel}, numbered with
$k\in\{1,2,3,4\}$, and we have that
$N'=3$, hence we get three sets of boundary conditions. Let us detail the
latter. 
  \\- For $p=1$, we have $j(1)=1$, $\tilde{\jmath}(1)=0$, and we see from the graph that $k_1=1$, 
  \\- for  $p=2$, we have $j(2)=2$, $\tilde{\jmath}(2)=1$, and we see from the graph that $k_1=2$, $k_2=4$, $k_3=3$, 
  \\- for  $p=3$, we have $j(3)=1$, $\tilde{\jmath}(3)=2$, and we see from the graph that $k_1=3$, $k_2=2$ and $k_3=4$.
  \\This yields three equations of the form \eqref{equ_Boundary} as
  follows, with $A_1(t)$ a scalar, $A_2(t)$   a $4\times4$ matrix
  and $A_3(t)$ a $3\times3$ matrix:
  \quad $v_1(t,0)=-A_1(t)\;i_1(t,0)$,
  \begin{equation*}
    \label{eq:ex-1}
      \begin{pmatrix} v_2(t,0)\\v_4(t,0)\\ v_1(t,1) \\  v_3(t,1) \end{pmatrix}
      =A_2(t)\!
      \begin{pmatrix} -i_2(t,0)\\-i_4(t,0)\\ i_1(t,1)\\
        i_3(t,1) \end{pmatrix} ,\ \ \ \ \text{and}\ \ 
      \begin{pmatrix} v_3(t,0)\\v_2(t,1)\\v_4(t,1) \end{pmatrix}
      =A_3(t)\!
      \begin{pmatrix} -i_3(t,0)\\i_2(t,1)\\i_4(t,1) \end{pmatrix}.
  \end{equation*}
\end{xmpl}
\begin{rmrk}[On the minus signs in the vector $I_p$ in \eqref{eq:defVI}]
  We shall see later why \eqref{eq:dissip-p} amounts to energy dissipation in some sense.
This is one justification for the minus signs in 
\eqref{eq:defVI}: removing these minus signs would transfer them into
$A_p(t)$ in \eqref{equ_Boundary}, and dissipativity
would then be expressed by an intricate property for the new $A_p(t)$  rather than
the natural strict positivity in \eqref{eq:dissip-p}. Alternatively, from a circuit-theoretic 
viewpoint, the minus
signs are justified by Kirchoff's law of currents.
\end{rmrk}
\begin{rmrk}[On the normalization of line lengths]
  We have assumed that the space variable $x$ belongs to the interval $[0,1]$ for every $k$ in
  equation \eqref{eq_tel} rather than $[0,\ell_k ]$ for some positive $\ell_k$. This is
  no loss of generality, for such a normalization can always be achieved by a linear change of variable on
  $x$. With this normalization, $\tau_k$ given by equation \eqref{eq:1} has the meaning of a time delay.
\end{rmrk}
\begin{rmrk}[On the possibility of loops]
  \label{rmk-loops}
In the above framework, nothing prevents an edge from being both  
outgoing from, and 
incoming to a given node $p$. In this case,  the index $k$ of this edge appears
twice in equation \eqref{eq:defVI}, once as a $k_j$ with $j\leq j(p)$ and 
once as a $k_{j(p)+l}$ with
$l\leq \tilde{\jmath}$.
\end{rmrk}

\medskip

So far, we endowed  a system consisting of $N$ PDE,
indexed by the edges of our graph (namely: \eqref{eq_tel}),
 with boundary
conditions given by   a collection  
of $N'$ linear time-dependent relations, indexed 
by the nodes of the graph (namely: \eqref{equ_Boundary}).
As a result, the boundary conditions at $x=0$ and at $x=1$ for a given
telegrapher's equation of the form \eqref{eq_tel} are generally obtained 
from two different 
relations of the form\eqref{equ_Boundary}.
To compactify the notation, we shall rewrite
the boundary conditions in lumped form, as a single
linear relation between concatenated vectors $\mathbf{V}(t) $ and $\mathbf{I}(t)$ defined by:
\begin{equation}
  \label{eq:7}
  \begin{split}
    v(x,t)=\begin{pmatrix} v_1(x,t)\\\vdots\\v_N(x,t) \end{pmatrix}
    &\,,\quad i(x,t)=\begin{pmatrix}
      i_1(x,t)\\\vdots\\i_N(x,t) \end{pmatrix} \,,
    \\
     &\hspace{-3em}\mathbf{V}(t)
    =\begin{pmatrix} \;\\[-1.1ex]v(t,0) \\\;\\ v(t,1)
      \\[-1.1ex]\; \end{pmatrix}
    \,,\quad \mathbf{I}(t)
    =\begin{pmatrix} \;\\[-1.1ex]-i(t,0) \\\;\\ i(t,1)
      \\[-1.1ex]\; \end{pmatrix} \,,
  \end{split}
\end{equation}
that aggregate \emph{all} boundary values of voltages and currents in the 
lines.
Since the concatenation of all vectors $V_p(t)$ (resp.\ $I_p(t)$) 
defined in \eqref{eq:defVI} 
contains each 
component of $\mathbf{V}(t)$ (resp.\ $\mathbf{I}(t)$) exactly  once, as
$p$ ranges from $1$ to $N'$, there is a $2N\!\times\! 2N$ permutation matrix $P_1$ such that
\begin{equation}
  \label{eq:8.1}
  \begin{pmatrix}
    V_1(t)\\\vdots\\V_{N'}(t)
  \end{pmatrix}
  =P_1\,\mathbf{V}(t)
  \,,\quad
  \begin{pmatrix}
    I_1(t)\\\vdots\\I_{N'}(t)
  \end{pmatrix}
  =P_1\,\mathbf{I}(t)\,.
\end{equation}
The set of equations \eqref{equ_Boundary}, $1\leq p\leq N'$, can now be written as
\begin{equation}
  \label{eq:9}
  \mathbf{V}(t)=\mathbf{A}(t)\,\mathbf{I}(t)
\end{equation}
with
\begin{equation}
  \label{eq:9.1}
  \mathbf{A}(t)=P_1^{-1}\, \diag(A_1(t),\ldots,A_{N'}) \,P_1\,
\end{equation}
where $\diag(A_1(t),\ldots,A_{N'}(t))$ is a block-diagonal $2N\times2N$ matrix.
This ``aggregated''
notation may be understood as collapsing all the nodes into a single one;  all edges are then ``loops'' as
described in Remark~\ref{rmk-loops}.
Clearly, the hypotheses on $A_p$ made in  \eqref{eq:dissip-p} 
translate into the following 
assumption on the matrix $\mathbf{A}(t)$ that will be used
throughout the paper :
\begin{assumption}
  \label{ass:dissip}
  The matrix-valued map $t\mapsto \mathbf{A}(t)$ is measurable  and essentially
bounded $[0,+\infty)\to\xR^{2N\times2N}$, and there exists a 
  real number $\alpha$, independent of $t$, such that $\alpha>0$ and
\begin{equation}
  \label{eq:dissip}
  \mathbf{A} (t)+\mathbf{A}^{*}(t) \geq \alpha\,Id,\qquad 
t\in\xR.
\end{equation}
\end{assumption}

\subsection{Well-posedness of the evolution problem in the $L^p$ and $C^0$ cases}
\label{sec-existence}

Equations \eqref{eq_tel} ($1\leq k\leq N$) and \eqref{eq:9}-\eqref{eq:7}
define
a linear time-varying dynamical system, whose state at time $t$ 
consists of a collection of $2N$ real functions on $[0,1]$, namely
$x\mapsto v_k(t,x)$ and $x\mapsto i_k(t,x)$ for $1\leq k\leq N$.
Before we can study the stability of this dynamical system,
we need to address  the issue of well-posedness, \textit{i.e.}\ of existence and uniqueness of solutions given initial
conditions $v_k(0,.)$ and $i_k(0,.)$ (the Cauchy problem).
When the matrices $A_p(t)$ (or equivalently the matrix $\mathbf{A}(t)$) 
do not actually depend on $t$, well-posedness results are classical,
see for instance the textbooks  \cite{bastin2016stability,Dafermos_2010}.

In the time-varying case, which is our concern here,
a very definition of well-posedness seems hard to find in the 
literature,  perhaps  because the introduction of  
time dependent boundary conditions 
leads to a failure of classical semigroup theory.
We shall consider  two cases according to whether the state  at time $t$
consists of continuous functions or merely
$L^p$-summable functions on $[0,1]$, $1\leq p\leq \infty$.

\medskip

Let us first fix notation. We denote respectively  
by $\xN$ and $\xR$ the sets of nonnegative integers and real numbers.
We write the Euclidean norm of $x\in\xR^l$ as $\|x\|$, and the 
Euclidean scalar product
of $x,y\in\xR^l$ as $\langle x,y\rangle$, irrespectively 
of $l$. 
 We put $C^0(E)$ for the space of real continuous functions on any 
topological space $E$.
When $E$ is compact we endow $C^0(E)$ with the \emph{sup} norm.
Also, whenever $E\subset \xR^l$ is measurable and $1\leq p<\infty$, we put
$L^p(E)$ for the familiar Lebesgue space of 
(equivalent classes of a.e. coinciding) real-valued
measurable functions on $E$  whose absolute value to the $p$\textsuperscript{th} power is 
integrable,
endowed with the norm $\|f\|_{L^p(E)}=(\int_E|f(x)|^p\mathrm{d}x)^{1/p}$
where $\mathrm{d}x$ indicates the differential of Lebesgue measure (restricted to $E$).
The space $L^\infty(E)$ corresponds to real, essentially bounded 
Lebesgue measurable functions, normed with the essential supremum
of their absolute value on $E$. 
More generally, for
$F$ a Banach space with norm $\|.\|_F$, we let $C^0(E,F)$ be the space of 
$F$-valued continuous functions on $E$, and if $E$ is compact we set
$\|f\|_{C^0(E,F)}=\sup_E\|f\|_F$. In a similar way,
$L^p(E,F)$ is the space of $F$-valued measurable functions $f$ on
$E$ such that $\|f\|_F\in L^p(E)$.
We also define locally integrable functions: $L^p_{loc}(E)$
designates the space of functions  whose 
restriction $f_{|K}$ to any compact set $K\subset E$ belongs to
$L^p(K)$. Likewise, we let 
$L^p_{loc}(E,F)$ be the space of $F$-valued measurable functions $f$ on
$E$ such that $\|f\|_F\in L^p_{loc}(E)$. Since $\xR^l$ is $\sigma$-compact,
the topology of
$L^p$-convergence  on every compact set is
 metrizable on $L^p_{loc}(\xR^l,F)$.
The spectral norm of a linear operator $B:F_1\to F_2$ between two Banach 
spaces is $\vertiii{B}=\sup_{x\in F_1}\|Bx\|_{F_2}/\|x\|_{F_1}$, keeping
the notation  independent of $F_1$, $F_2$ for simplicity.

\medskip

Next, let us make  precise the meaning of  \eqref{eq_tel} and \eqref{eq:9} 
when $v_k$ and $i_k$
lie in 
$L^1_{loc}(\overline{\Omega})$, where $\Omega$ is defined by \eqref{eq:Omega}
and $\overline{\Omega}$ indicates the closure of $\Omega$ in $\xR^2$.
Later, we shall see this space is 
big enough to accomodate cases we have in mind.
Note that $\overline{\Omega}=[0,\infty)\times[0,1]$, and that
$L^1_{loc}(\overline{\Omega})$ identifies with 
a subspace of $L^1_{loc}(\Omega)$, since
$([0,\infty)\times[0,1])\setminus \Omega$ has 2-D Lebesgue
measure zero. Indeed, the latter set is just the boundary $\partial\Omega$
of $\Omega$ in $\xR^2$:
\begin{equation}
\label{bOmega}
\partial\Omega=\left(\{0\}\!\times\!(0,1)\right)\cup\left([0,+\infty)\!\times\!\{0\}\right)\cup\left([0,+\infty)\!\times\!\{1\}\right).
\end{equation}

Equation \eqref{eq_tel} is understood in the distributional
sense as soon as $(v_k,i_k)\in L^1_{loc}(\Omega)\times L^1_{loc}(\Omega)$.
That is,  $(v_k,i_k)$
is a solution to \eqref{eq_tel} if, 
for all $C^\infty$-smooth functions
$\varphi:\Omega\to\xR$ with compact support, it holds that
\begin{equation}
  \label{equ_sol_continu_tel}
  \begin{split}
    &\int\hspace{-.75em}\int_{\Omega} \Bigl(L_k \, i_k(t,x) \frac{\partial \varphi}{\partial
      t}(t,x) + v_k(t,x) \frac{\partial \varphi}{\partial x}(t,x)\Bigr)\mathrm{d}t\mathrm{d}x=0\,,\\
    &\int\hspace{-.75em}\int_{\Omega} \Bigl( C_k\, v_k(t,x) \frac{\partial \varphi}{\partial
      t}(t,x) + i_k(t,x) \frac{\partial \varphi}{\partial x}(t,x)\Bigr)\mathrm{d}t\mathrm{d}x=0\,.
  \end{split}
\end{equation}
As to \eqref{eq:9}, the definition \eqref{eq:7} of 
$\mathbf{V}$ and $\mathbf{I}$, 
as well as the choice of initial conditions $v_k(0,.)$ and
$i_k(0,.)$, require that $v_k$ and $i_k$ extend
in some way to $\partial\Omega$ described in \eqref{bOmega}, and this is 
where their membership to $ L^1_{loc}(\overline{\Omega})$ (not just to
$L^1_{loc}(\Omega)$) is useful.
In fact, when  $h\in L^1_{loc}(\overline{\Omega})=L^1_{loc}([0,\infty)\times[0,1])$, 
we get from Fubini's theorem
that $\tau\mapsto h(\tau,x)$ belongs to
$L^1_{loc}([0,\infty))$ for a.e. $x\in[0,1]$ and that
$s\mapsto h(t,s)$  lies in $L^1([0,1])$ for a.e. $t\in[0,\infty)$.
For such $x$ and $t$, we set
\begin{equation}
    \label{eq:4}
    \begin{split}
      \widehat{h}(t,0)=\lim_{\varepsilon\to0}&\frac{1}{\varepsilon}
      \int_0^\varepsilon h(t,s)\mathrm{d}s\,,\ \ \ \
      \widehat{h}(t,1)=\lim_{\varepsilon\to0}\frac{1}{\varepsilon}
      \int_{1-\varepsilon}^1 h(t,s)\mathrm{d}s\,,
      \\
      &\widehat{h}(0,x)=\lim_{\varepsilon\to0}\frac{1}{\varepsilon}
      \int_0^\varepsilon h(s,x)\mathrm{d}s\,, \ \ \ \
      \text{whenever the limits exist. }
    \end{split}
\end{equation}
%
%
\begin{definition}
    \label{def:prolong}
   We say that 
$h\in L^1_{loc}(\overline{\Omega})$ 
\emph{has a strict extension to
  $\partial\Omega$} if and only if the limits in \eqref{eq:4} exist for
  almost all $x\in(0,1)$ and almost all $t\in(0,\infty)$, and then the functions
  $x\mapsto\widehat{h}(0,x)$, $t\mapsto\widehat{h}(t,0)$ and 
$t\mapsto\widehat{h}(t,1)$ define the strict extension of $h$ to $\partial\Omega$,
almost everywhere with respect to $\mathcal{H}^1$-Hausdorff measure\footnote{ 
  See {\it e.g.} \cite[ch. 2]{EvaGar1992} for the definition of
  Hausdorff measures.
  Here, $\mathcal{H}^1$ restricted to $\partial\Omega\subset\xR^2$ is simply
  the measure whose restriction to each curve  $\{0\}\!\times\!(0,1)$,
  $[0,+\infty)\!\times\!\{0\}$ and $[ 0,+\infty)\!\times\!\{1\}$
  coincides with arc length.
}.
\end{definition}
\begin{rmrk}\label{rmk:prolong-continu}
Definition \ref{def:prolong} may look strange at first glance, 
since when $h\in L^1_{loc}(\overline{\Omega})$ it seems
to be defined already on 
$\partial\Omega\subset\overline{\Omega}$; but of  course it is not so, because
$\partial\Omega$ has 2-D Lebesgue measure zero, hence the values assumed 
by $h$ there are immaterial. When the limits in \eqref{eq:4} exist 
for a.e. $x$ and $t$, they produce a specific definition of $h$ 
 on $\partial\Omega$, a.e. with respect to $\mathcal{H}^1$, that we call the
strict extension.
If $h:\Omega\to\xR$ is continuous and extends continuously 
$\overline \Omega\to\xR$, clearly the strict extension exists and it is the natural one.
Even then, we sometimes
use the notation $\widehat{h}(0,x)$, $\widehat{h}(t,0)$ and
  $\widehat{h}(t,1)$ for reasons of consistency, although 
writing $h(0,x)$, $h(t,0)$ and $h(t,1)$ is more appropriate in 
this case. 
\end{rmrk}
\noindent
If all  $v_k$ and $i_k$ have a strict extension to $\partial\Omega$, then we
interpret the boundary conditions \eqref{eq:9}
to mean the following set of equalities between (a.e. defined) 
measurable functions of a single variable $t$:
\begin{equation}
  \label{eq:6}
  \begin{pmatrix} \widehat{v}_1(t,0)\\\vdots\\\widehat{v}_N(t,0)\\\widehat{v}_1(t,1)\\\vdots\\\widehat{v}_N(t,1)  \end{pmatrix}
  =\mathbf{A}(t)
    \begin{pmatrix} -\widehat{\imath}_1(t,0)\\\vdots\\-\widehat{\imath}_N(t,0)\\\widehat{\imath}_1(t,1)\\\vdots\\\widehat{\imath}_N(t,1)  \end{pmatrix}
    ,\qquad \text{a.e. }t\in(0,\infty).
\end{equation}

\medskip

We can now state a well-posedness result for System \eqref{eq_tel}-\eqref{eq:9}. Part I deals with solutions belonging to  $L^1_{loc}([0,\infty),L^p([0,1]))$, and part 
II is about continuous solutions.
They do not run completely parallel to each other, because continuity requires a compatibility relation on the initial
conditions, 
see \eqref{eq:203}.
The theorem is standard in nature but, as mentioned already, we could not find 
a reference in the
literature for the case of \emph{time-varying} boundary conditions 
\eqref{eq:6}. To connect the statement with the previous discussion, we observe that $L^1_{loc}([0,\infty),L^p([0,1]))\subset L^1_{loc}([0,\infty)\times[0,1])=L^1_{loc}(\overline{\Omega})$ for $1\leq p\leq\infty$,
by H\"older's inequality and Fubini's theorem.

\begin{theorem}[Well-posedness]
  \label{existence_sol_tele}
Let  $\mathbf{A}:[0,\infty)\to\xR^{2N\times 2N}$ meet Assumption \ref{ass:dissip} and $1\leq p\leq\infty$.
$\ $\\\textbf{\textup{I)}}
If $\InitCond{i}_k,\InitCond{v}_k\in L^p([0,1])$, $1\leq k\leq N$, there is a unique map
$(t,x)\mapsto(v_1(t,x),\ldots,$ $v_N(t,x),i_1(t,x),\ldots,i_N(t,x))$
from  $\Omega$ into $\xR^{2N\times 2N}$ such that:
\\ $\bullet$\quad
$t\mapsto(v_1(t,.),\ldots,v_N(t,.),i_1(t,.),\ldots,i_N(t,.))$ belongs to
$L^1_{loc}([0,\infty),(L^p([0,1]))^{2N})$ and $v_k$, $i_k$ have
a strict extension to $\partial\Omega$ satisfying the initial conditions
\begin{equation}
  \label{eq:202}
  \widehat{v}_k(0,x)=\InitCond{v}_k(x),\,\ \widehat{\imath}_k(0,x)=\InitCond{i}_k(x)\,\ \ k=0,\ldots,N\,,
\end{equation}
$\bullet$\quad
$(t,x)\mapsto(v_1(t,x),\ldots,v_N(t,x),i_1(t,x),\ldots,i_N(t,x))$
is a solution of \eqref{eq_tel}-\eqref{eq:9},
$1\leq k\leq N$, in the sense of \eqref{equ_sol_continu_tel} and \eqref{eq:6}.

\smallskip

\noindent
\textbf{\textup{II)}}
If, in addition, $t\mapsto \mathbf{A}(t)$ is continuous and 
$\InitCond{v}_1,\ldots,\InitCond{v}_N,\InitCond{i}_1,\ldots,\InitCond{i}_N$
are elements of
$C^0([0,1])$ satisfying
\begin{equation}
  \label{eq:203}
  \begin{pmatrix} \InitCond{v}_1(0)\\\vdots\\\InitCond{v}_N(0)\\\InitCond{v}_1(1)\\\vdots\\\InitCond{v}_N(1)  \end{pmatrix}
  =\mathbf{A}(0)
    \begin{pmatrix} -\InitCond{i}_1(0)\\\vdots\\-\InitCond{i}_N(0)\\\InitCond{i}_1(1)\\\vdots\\\InitCond{i}_N(1)  \end{pmatrix}
    \,,
\end{equation}
then the map
$(t,x)\mapsto(v_1(t,x),\ldots,v_N(t,x),i_1(t,x),\ldots,i_N(t,x))$ from
part {\bf I} is continuous
$\overline{\Omega}\to\mathbb{R}^{2N}$ (equivalently:
$t\mapsto(v_1(t,.),\ldots,v_N(t,.),i_1(t,.),\ldots,i_N(t,.))$ is continuous
$[0,\infty)\to C^0([0,1])^{2N\times 2N}$) and 
satisfies the initial conditions in the strong sense:
\begin{equation}
  \label{eq:204}
  v_k(0,x)=\InitCond{v}_k(x),\,\ i_k(0,x)=\InitCond{i}_k(x)\,,
\qquad x\in[0,1],\quad k=0,\ldots,N\,.
\end{equation}
\end{theorem}
\begin{rmrk}
  \label{rmk-hypWellPosed}
  Assumption \ref{ass:dissip} is stronger than needed
 for the previous result to hold. In fact,
it is enough for part I that the maps $t\mapsto \mathbf{A}(t)$
  and $t\mapsto\left(I+\mathbf{A}(t)\,\mathbf{K}\right)^{-1}$ 
  be well defined, measurable and bounded ($\mathbf{K}$ is  defined in \eqref{eq:9.5}), and for part II
that they be  continuous and bounded.
  We do not dwell on such generalizations.
\end{rmrk}

The proof of Theorem \ref{existence_sol_tele} is given at the end of  
Section~\ref{equivalence_telegrapher_delay},
after establishing the equivalence of \eqref{eq_tel}-\eqref{eq:9}
with a suitable difference-delay system.
As a first step in this direction, we stress below
the special form  of solutions to \eqref{eq_tel} in $L^p_{loc}(\Omega)$, and
show they have a strict 
extension to $\partial\Omega$ if, moreover, they lie in 
$L^p_{loc}(\overline{\Omega})$.

\begin{proposition}
  \label{prop:resolu}
  Let $i_k $ and $v_k$ belong to $L^p_{loc}(\Omega)$ (resp.\ $C^0(\Omega)$)
for some $p\in[1,\infty]$, and satisfy
  \eqref{eq_tel} on $\Omega$ in the sense of \eqref{equ_sol_continu_tel}. 
Then, the following three properties hold.

 \textup{ (\textit{i})}
    There exists two functions $f_k$ and $g_k$ in $L^p_{loc}((-\infty,1))$ and
$L^p_{loc}((0,\infty))$ (resp.\ in  $C^0((-\infty,1))$ and
$C^0((0,\infty))$) such that
    \begin{equation}
      \label{eq:0}
      v_k(t,x)=f_k(x-\frac{t}{\tau_k})+g_k(x+\frac{t}{\tau_k})
      ,\ \ 
      i_k(t,x)=K_k\left(f_k(x-\frac{t}{\tau_k})-g_k(x+\frac{t}{\tau_k})\right)
      , 
    \end{equation}
for almost every (resp.\ every) $(x,t)$ in $\Omega$, where $\tau_k,K_k$ are
defined by \eqref{eq:1}.

 \textup{ (\textit{ii})}
If, in addition, $v_k$ and $i_k$ lie in $L^p_{loc}(\overline{\Omega})$
(resp.\ extend continuously $\overline{\Omega}\to\xR$), then $f_k$ and $g_k$ 
lie in $L^p_{loc}((-\infty,1])$ and
$L^p_{loc}([0,\infty))$ (resp.\ in  $C^0((-\infty,1])$ and $C^0([0,\infty))$),
moreover $v_k$, $i_k$ have a strict extension to $\partial\,\Omega$ according to
  Definition~\ref{def:prolong}. More precisely, we have that
\begin{subequations}
    \begin{align}
      \label{eq:15a}
      &\widehat{v}_k(t,0)=f_k(-\frac{t}{\tau_k})+g_k(\frac{t}{\tau_k}),  
      &&\widehat{\imath}_k(t,0)=K_k\left(f_k(-\frac{t}{\tau_k})-g_k(\frac{t}{\tau_k})\right), \\
      \label{eq:15b}
      &\widehat{v}_k(t,1)=f_k(1\!-\!\frac{t}{\tau_k})+g_k(1\!+\!\frac{t}{\tau_k}),  
      &&\widehat{\imath}_k(t,1)=K_k\left(f_k(1\!-\!\frac{t}{\tau_k})-g_k(1\!+\!\frac{t}{\tau_k})\right),
      \\
      \label{eq:15c}
      &\widehat{\imath}_k(0,x)=K_k\left(f_k(x)-g_k(x)\right),  
      &&\widehat{v}_k(0,x)=f_k(x)+g_k(x),
    \end{align}
\end{subequations}
  where \eqref{eq:15a} and \eqref{eq:15b} hold for almost all (resp.\ all) $t$ in $(0,+\infty)$ and 
 \eqref{eq:15c} for almost all (resp.\ all) $x$ in $(0,1)$.

 \textup{(\textit{iii})}  Conversely, if $f_k$ and $g_k$ lie in $L^p_{loc}((-\infty,1))$ and
$L^p_{loc}((0,\infty))$ (resp.\ in  $C^0((-\infty,1))$ and $C^0((0,\infty))$),
then $v_k$ and $i_k$ given by \eqref{eq:0} belong to $L^p_{loc}(\Omega)$ 
(resp.\ $C^0(\Omega)$) and satisfy \eqref{eq_tel}. If, moreover, $f_k$ and $g_k$ lie in $L^p_{loc}((-\infty,1])$ and
$L^p_{loc}([0,\infty))$ (resp.\ in  $C^0((-\infty,1])$ and $C^0([0,\infty))$),
then $v_k$ and $i_k$ belong to $L^p_{loc}(\overline{\Omega})$ (resp.\ $C^0(\overline{\Omega})$) and \eqref{eq:15a}--\eqref{eq:15c} hold.
\end{proposition}
\begin{proof}
The proof of point (\textit{i}) rests on a linear change of
variables, valid even in the distributional sense: if we introduce new
variables $r=x-t/\tau_k$, $s=x+t/\tau_k$ and new functions
$f_k$, $g_k$ on $\Omega_1:=\{(r,s)\in\xR^2,\;0<r+s<2\text{ and } -\infty <r-s<0\}$, {\it via}
\begin{equation*}
  \label{eq:resol1}
  \begin{split}
    \begin{pmatrix}
      f_k(r,s)\\g_k(r,s)
    \end{pmatrix}
    &=
    \frac12
    \begin{pmatrix}
        1 & 1/K_k\\ 1 & -1/K_k
    \end{pmatrix}
    \begin{pmatrix}
       v_k \bigl(\frac{\tau_k}{2} (-r+s)\,,\frac12(r+s)\bigr) \\[.7ex]
      i_k \bigl(\frac{\tau_k}{2} (-r+s)\,,\frac12(r+s)\bigr)
    \end{pmatrix}
    ,\\
    \begin{pmatrix}
      v_k(t,x)\\i_k(t,x)
    \end{pmatrix}
    &=
    \begin{pmatrix}
      1\!\!\!\!\!\! & 1 \\ K_k\!\!\!\!\!\! & -K_k
    \end{pmatrix}
    \begin{pmatrix}
      f_k\bigl(x-t/\tau_k,x+t/\tau_k\bigr)\\g_k\bigl(x-t/\tau_k,x+t/\tau_k\bigr)
    \end{pmatrix}
   ,
  \end{split}
\end{equation*}
 then $f_k,g_k$ are in one-to-one correspondence with $v_k$, $i_k$, they
  are in $L^p_{loc}(\Omega_1)$ (resp.\ $C^0(\Omega_1,\xR)$) if and only if the latter are in $L^p_{loc}(\Omega)$
  (resp.\ $C^0(\Omega,\xR)$),
and  System \eqref{eq_tel} gets   transformed into the distributional
identity $\partial f_k/\partial s=\partial g_k/\partial r=0$.
This equation means
that $f_k$ does not depend on the second argument nor
$g_k$ on the first one, hence the form \eqref{eq:0} for $v_k$ and
$i_k$.

  We turn to Point~(\textit{ii}). 
First, we observe that if
$v_k$ and $i_k$ lie in $L^p_{loc}(\overline{\Omega})$
(resp.\ extend continuously $\overline{\Omega}\to\xR$), then $f_k$ and $g_k$ 
lie in $L^p_{loc}((-\infty,1])$ and
$L^p_{loc}([0,\infty))$ (resp.\ in  $C^0((-\infty,1])$ and $C^0([0,\infty))$),
by the change of variable formula (resp.\ by inspection). 
The case where 
$v_k$ and $i_k$ extend continuously $\overline{\Omega}\to\xR$ is now
obvious. To handle the case where
$v_k,i_k\in L^p_{loc}(\overline{\Omega})$, recall 
that a (non-centered) Lebesgue point of a function $\ell\in L^1_{loc}(\xR)$ 
is a point $x\in\xR$ such
  that $\lim_{|I|\to 0, I\ni x}\frac{1}{|I|}\int_I|\ell(y)-\ell(x)|\mathrm{d}y=0$, where the limit is taken over all closed intervals $I$
containing $x$ and $|I|$ indicates the length of $I$.  
Let $\widetilde{f}_k$ and $\widetilde{g}_k$ be the extensions by $0$ of $f_k$ and $g_k$ to the whole real line. Using \eqref{eq:0} in \eqref{eq:4}, we see 
that \eqref{eq:15a} certainly holds for $t\in(0,\infty)$ 
such that
  $-t/\tau_k$ is a Lebesgue point of $\widetilde{f}_k$ and $t/\tau_k$ is a Lebesgue point of $\widetilde{g}_k$,
    \eqref{eq:15b} if
  $1-t/\tau_k$ is a Lebesgue point of $\widetilde{f}_k$ and $1+t/\tau_k$ is a Lebesgue point of $\widetilde{g}_k$, and
  \eqref{eq:15c} if $x\in(0,1)$ is a Lebesgue point of both $\widetilde{f}_k$ and $\widetilde{g}_k$. Since almost all points are
  Lebesgue points of a given function in $L^1_{loc}(\xR)$~\cite[thm. 1.34]{EvaGar1992}, while $L^p_{loc}(\xR)\subset L^1_{loc}(\xR)$ by H\"older's inequality,
  this proves Point (\textit{ii}).
\quad
Point (\textit{iii}) is obvious, reverting computations.
\end{proof}

\begin{rmrk}
\label{remc}
The weak formulation \eqref{equ_sol_continu_tel} defines solutions $v_k$, $i_k$
to  \eqref{eq_tel} as locally integrable functions
 $\Omega\to\xR$,
while Theorem \ref{existence_sol_tele} stresses their representation as 
functions $[0,+\infty)\to L^p([0,1])$.
The two points of view  are essentially  equivalent by Fubini's theorem, but 
suggestive of different moods. In this connection, it is worth 
mentioning that if $p<\infty$, then the solution set forth in Part I of 
Theorem \ref{existence_sol_tele}
not only belongs to $L^1_{loc}([0,\infty),(L^p([0,1]))^{2N})$,
but  in fact is continuous $[0,\infty)\to(L^p([0,1]))^{2N\times 2N}$.
Indeed, granted that $f_k$ and $g_k$ lie in
$L^p_{loc}((-\infty,1])$ and $L^p_{loc}([0,\infty))$ by Proposition
\ref{prop:resolu},
this  follows from the very proof of the theorem ({\it cf.} 
\eqref{eq:11} and \eqref{eq:3} below)
and the fact that $\tau\mapsto f(.-\tau)$ is continuous $\xR\to L^p(\xR)$,
whenever $f\in L^p(\xR)$, $p<\infty$. 
\end{rmrk}

\subsection{Difference-delay equations and their relation with networks of telegrapher's equations}
\label{equivalence_telegrapher_delay}

A general linear time-varying difference-delay equation\footnote{
  There is no consensus in the literature on a name for
  equations like \eqref{syst_delay_generique}.
  We stick with the term ``difference-delay equations'' (or
  ``difference-delay systems'') throughout the present paper.
}
in the variable $z$ is of the form
 \begin{eqnarray}  
  \label{syst_delay_generique}
z(t)=\sum\limits_{i=1}^\newN D_i(t) \,z(t-\eta_i)\ \ \ \text{for all (or almost
   all) }t\geq0\,,
 \end{eqnarray}
where the delays $0<\eta_1\leq\cdots\leq\eta_\newN$ are arranged in nondecreasing  order, each $t\mapsto
D_i(t)$ is a $d\times d$ matrix-valued function, and solutions $t\mapsto z(t)$ are
$\xR^d$-valued functions.
Hereafter, we make the following assumption.
\begin{assumption}
  \label{ass:borne}
  The maps $t\mapsto D_i(t)$ in \eqref{syst_delay_generique} all belong to $L^\infty([0,+\infty), \xR^{d\times d})$.
\end{assumption}
Given initial conditions on $[-\eta_\newN,0]$,
we recap existence and uniqueness of solutions to 
\eqref{syst_delay_generique} in the following
theorem.
The existence of continuous solutions 
requires an additional  continuity assumption 
on the $D_i$, as well as compatibility relations on the initial 
conditions; this is why we introduce the following space: 
\begin{equation}
  \label{eq:2}
  \mathcal{C}:=\{\phi \in C^0([-\eta_\newN,0], \mathbb{R}^{d})\,|\,\phi(0)=\sum\limits_{i=1}^\newN D_i(0) \phi(-\eta_i)\}\,.
\end{equation}
\begin{theorem}
\label{existence_solution_delay}
Let Assumption \ref{ass:borne} hold and $\phi$ be an element of $L^p([-\eta_\newN,0],$ $\mathbb{R}^{d})$ 
with $1\leq p\leq\infty$. 
\begin{itemize}
\item[\textup{(\textit{i})}]
There is a unique solution $z$ to \eqref{syst_delay_generique}
in  $L^p_{loc}([-\eta_\newN,+\infty),\mathbb{R}^{d})$
meeting the initial condition $z_{|[-\eta_\newN,0]}=\phi$.
\item[\textup{(\textit{ii})}] If, moreover,
the maps $D_i:[0,+\infty)\to\xR^{d\times d}$ are continuous and
$\phi\in \mathcal{C}$, then $z\in C^0([-\eta_\newN,+\infty),\mathbb{R}^{d}))$.
\end{itemize}\end{theorem}
\begin{proof}
 This is  a classical,  elementary inductive argument, see \textit{e.g.}\ \cite{Hale}: for any $T\geq0$, 
 if a solution has been found on $[-\eta_\newN,T]$, it clearly can be extended to $[-\eta_\newN,T+\eta_1]$ in a unique manner using 
\eqref{syst_delay_generique}.
When the $D_i(.)$ are continuous, $\phi\in \mathcal{C}$ is clearly necessary and sufficient for the
unique solution to be continuous.
\end{proof}
\begin{rmrk}[merging repeated delays]
  \label{rmk-repeated-delays}
  In \eqref{syst_delay_generique}, we allow for repeated delays, \textit{i.e.} it may be that $\eta_i=\eta_{i+1}$ for some $i$.
  This to comply with \eqref{eq_tel}-\eqref{eq:1}, where it
  would be too restrictive to require that the numbers $\tau_k$ are distinct, and because we are
  about to  convert  \eqref{eq_tel}-\eqref{eq:1} into \eqref{syst_delay_generique} in such a way
  that $\eta_i=\tau_i$. 
  However, when dealing with \eqref{syst_delay_generique}, it is better to avoid repetition by merging terms with the same delay.
  Since it will be needed in the statement of Theorem~\ref{theorem-central-delay}, let us formalize
  this:  first, define an enumeration without repetition of the original list of delays, say,
$0<\widehat{\eta}_1<\widehat{\eta}_2<\cdots<\widehat{\eta}_{\widehat{\newN}}$ with
  $\widehat{\newN}\leq \newN$, then define for each $j$
  \begin{equation}
    \label{eq:21}
    \widehat{D}_j(t)=\sum_{\resizebox{3.6em}{!}{$\{i,\, \eta_i={\widehat{\eta}}_j\}$}}D_i(t)\,. 
  \end{equation}
  It is clear that \eqref{syst_delay_generique}  can be re-written as
  $z(t)=\sum_{j=1}^{\widehat{\newN}}\widehat{D}_j(t) \, z(t-\widehat{\eta}_j)$,
  and if the $\eta_i$ were distinct already, then the system is left unchanged.
\end{rmrk}
We now construe the system of coupled telegrapher's equations from
Sections~\ref{sec:network}
and \ref{sec-existence} as a difference-delay system of the form 
\eqref{syst_delay_generique}. For this, let 
$(v_k,i_k)\in L^1_{loc}(\overline{\Omega})\times L^1_{loc}(\overline{\Omega})$ 
(resp.\ $C^0(\overline{\Omega})\times C^0(\overline{\Omega})$) be, for $1\leq k\leq N$,  solutions of \eqref{eq_tel}-\eqref{eq:9},
observing from Proposition \ref{prop:resolu} (\textit{ii}) that
the boundary conditions \eqref{eq:9} indeed make sense. 
Let $f_k$, $g_k$ be as in Proposition~\ref{prop:resolu}, and define:
\begin{equation}
  \label{eq:3}
x_k(t)=f_k(-\frac{t}{\tau_k})\ \ \text{and}\ \ y_k(t)=g_k(1+\frac{t}{\tau_k}).
\end{equation}
The functions  $f_k$ and $g_k$ lie in
 $L^1_{loc}((-\infty,1])$ and
$L^1_{loc}([0,\infty))$ (resp.\ in  $C^0((-\infty,1])$ and $C^0([0,\infty))$)
by Proposition~\ref{prop:resolu},
therefore $x_k$ and $y_k$ 
lie in $L^1_{loc}([-\tau_k,+\infty))$ (resp.\ $C^0([-\tau_k,+\infty))$).
Moreover, the boundary values of  $v_k$ and $i_k$ are related to $x_k$ and $y_k$ as
follows (substitute \eqref{eq:3} in \eqref{eq:15a} and  \eqref{eq:15b}):
\begin{eqnarray}
\label{eq_resolved_form}
\left\{
 \begin{array}{llll}
\widehat{v}_k(t,0)=x_k(t)+y_k(t-\tau_k)\,,\\
\widehat{\imath}_k(t,0)=K_k[x_k(t)-y_k(t-\tau_k)]\,,\\
\widehat{v}_k(t,1)=x_k(t-\tau_k)+y_k(t)\,,\\
\widehat{\imath}_k(t,1)=K_k[x_k(t-\tau_k)-y_k(t)]\,.
\end{array}
\right.
\end{eqnarray}
Plugging \eqref{eq_resolved_form} in \eqref{eq:6} gives us
\begin{displaymath}
  \begin{pmatrix}
    x_1(t)\\\vdots \\x_N(t)\\y_1(t)\\\vdots\\ y_N(t)
  \end{pmatrix}
     +
  \begin{pmatrix}
    y_1(t-\tau_1)\\\vdots \\y_N(t-\tau_N)\\x_1(t-\tau_1)\\\vdots\\ x_N(t-\tau_N)
  \end{pmatrix}
=
  \mathbf{A}(t)\,\left[
  \begin{pmatrix}
    -K_1\,x_1(t)\\\vdots \\-K_N\,x_N(t)\\-K_1\,y_1(t)\\\vdots\\ -K_N\,y_N(t)
  \end{pmatrix}
     +
  \begin{pmatrix}
    K_1\,y_1(t-\tau_1)\\\vdots\\ K_N\,y_N(t-\tau_N)\\K_1\,x_1(t-\tau_1)\\\vdots\\ K_N\,x_N(t-\tau_N)
  \end{pmatrix}
  \right].
\end{displaymath}
Thus, if we define 
\begin{eqnarray}
\label{eq:9.5}
 \mathbf{K}=\diag(K_1,\ldots,K_N,K_1,\ldots,K_N), \qquad
P_2=\begin{pmatrix}
0\ \ Id\\Id\ \ 0
\end{pmatrix}
\end{eqnarray}
where $Id$ has size $N\times N$, and  observe that
$P_2\mathbf{K}=\mathbf{K}P_2$ while noting that
relation $\mathbf{K}=\mathbf{K}^*>0$ together
with the dissipativity condition \eqref{eq:dissip} 
entail that $I+\mathbf{A}(t)\,\mathbf{K}$ is invertible, we obtain:
\begin{equation}
  \label{eq:10}
  \begin{pmatrix}
    x_1(t)\\\vdots \\x_N(t)\\y_1(t)\\\vdots\\ y_N(t)
  \end{pmatrix}
  =
  -
  \left(I+\mathbf{A}(t)\,\mathbf{K}\right)^{-1}
  \left(I-\mathbf{A}(t)\,\mathbf{K}\right)P_2
  \begin{pmatrix}
    x_1(t-\tau_1)\\\vdots\\ x_N(t-\tau_N)\\y_1(t-\tau_1)\\\vdots\\ y_N(t-\tau_N)
  \end{pmatrix}
  \,.
\end{equation}
Setting $d=2N$ and letting
$z(t)$ be the vector $[x_1(t), \cdots,x_N(t)$, $y_1(t),\cdots,y_N(t)]^*$ and, for each $i \in
\{1,\cdots,,N\}$, the $d\times d$ matrix  $D_i(t)$ have the same $i$\textsuperscript{th} and $(i+N)$\textsuperscript{th} columns as the matrix
$-\left(I+\mathbf{A}(t)\,\mathbf{K}\right)^{-1}\left(I-\mathbf{A}(t)\,\mathbf{K}\right)P_2$,
the other columns being zero,
it is obvious that system~\eqref{eq:10}  can be rewritten in the form
\eqref{syst_delay_generique} with $\newN=N$ and $\eta_i=\tau_i$, $1\leq i\leq N$.
As for initial conditions, we obtain from \eqref{eq:3} and 
\eqref{eq:15c} that, for $t$ in $[-\tau_k,0]$,
\begin{equation}
  \label{eq:12}
  x_k(t)=\frac12\InitCond{v}_k\bigl(\frac{-t}{\tau_k}\bigr) +\frac1{2K_k}
    \InitCond{i}_k\bigl(\frac{-t}{\tau_k}\bigr)
  \,,\ 
  y_k(t)=\frac12 \InitCond{v}_k\bigl(1\!+\!\frac{t}{\tau_k}\bigr) -
    \frac1{2K_k}\InitCond{i}_k\bigl(1\!+\!\frac{t}{\tau_k}\bigr).
\end{equation}
Note that both $-t/\tau_k$ and $1+t/\tau_k$ range over
$[0,1]$ when $t$ ranges over  
$[-\tau_k,0]$. The only difference with the situation in Theorem
\ref{existence_solution_delay} is that initial values for $x_k$, $y_k$ are only
provided over $[-\tau_k,0]$ through \eqref{eq:3} and \eqref{eq:15c},
not over $[-\tau_N,0]$.
However, with the previous definitions of $z(t)$ and $D_i(t)$, $1\leq i\leq N$,
the values of $x_k$ and $y_k$ on $[-\tau_N,-\tau_k)$ when $\tau_k<\tau_N$ are
unimportant to the dynamics of  \eqref{syst_delay_generique} for $t\geq0$, 
because the columns of $D_i(t)$ other than 
$i$\textsuperscript{th} and $(i+N)$\textsuperscript{th} are identically zero. 
Thus, we may pick initial conditions for $x_k$ and $y_k$ on
 $[-\tau_N,-\tau_k)$ arbitrarily,
provided  that we comply with summability or continuity 
requirements. For instance, we can
extend $x_k$ and $y_k$ to $[-\tau_N,0]$ using the operators
$J_{[-\tau_k,0]}^{[-\tau_N,0]}$ 
defined as follows.
For $a<b<c$ three real numbers, $J_{[a,b]}^{[a,c]}$ be an extension 
operator mapping functions on $[a,b]$ to functions on $[a,c]$ so that 
 $L^p([a,b])$ gets mapped into $L^p([a,c])$ and $C^0([a,b])$ into 
$C^0([a,c])$, in a continuous manner. 
Such an operator is easily constructed by choosing a smooth 
function $\varphi:\xR\to\xR$ which is $1$ on $(-\infty,b]$ and $0$
on $[\min\{2b-a,c\},+\infty)$; then, for $f:[a,b]\to\xR$, define
$J_{[a,b]}^{[a,c]}f$ to be $f$ on $[a,b]$ and $\varphi(t)f(2b-t)$ for 
$t\in(b,c]$, where the product is interpreted as zero if $2b-t<a$.
Similarly we define $J_{[b,c]}^{[a,c]}$ mapping functions on $[b,c]$ to functions
on $[a,c]$.

We have now reduced the Cauchy problem for 
\eqref{eq_tel}-\eqref{eq:9}, $1\leq k\leq N$, 
to the Cauchy problem for a particular equation of the form 
\eqref{syst_delay_generique}. Moreover, it is 
obvious from what precedes that initial conditions 
in $L^p([0,1],\xR)$ (resp.\ $C^0([0,1])$  meeting \eqref{eq:203})  
for $v_k$, $i_k$ correspond
to initial conditions 
in $L^p([-\tau_N,0],\xR^{2N})$ (resp.\ $\mathcal{C}$) for $z$, and 
that solutions 
$v_k$, $i_k$ in $L^p_{loc}([0,\infty), L^p([0,1]))$
(resp.\ $C^0([0,\infty),C^0([0,1]))$)
correspond to solutions
$z\in L^p_{loc}([0,\infty),\xR^{2N})$ 
(resp.\ $C^0([0,\infty),\xR^{2N})$).

\begin{proof}[Proof of Theorem~\ref{existence_sol_tele}]
The above discussion (starting after
Theorem \ref{existence_solution_delay})
shows that the function 
  $(t,x)\mapsto(v_1(t,x),\ldots$, $v_N(t,x),i_1(t,x),\ldots,i_N(t,x))$ is a solution of  \eqref{eq_tel}-\eqref{eq:9}-\eqref{eq:202} for Part I or \eqref{eq_tel}-\eqref{eq:9}-\eqref{eq:204} for part II if and only if 
\begin{equation}
    \label{eq:11}
    \begin{split}
      & v_k(t,x)=x_k(t-x\tau_k)+y_k((x-1)\tau_k+t) \,, \\
      &i_k(t,x)=K_k\left(x_k(t-x\tau_k)-y_k((x-1)\tau_k+t)\right),
    \end{split}
\end{equation}
  where $t\mapsto(x_1(t),\ldots,x_N(t),y_i(t),\ldots,y_N(t))$ is  a solution of the difference-delay system \eqref{eq:10} in $L^p_{loc}([0,\infty),\xR^{2N})$ or
in $C^0([0,\infty),\xR^{2N})$,
  with initial conditions given by \eqref{eq:12} and extended if necessary to
$[-\tau_N,0]$ using the operator $J_{[-\tau_k,0]}^{[-\tau_N,0]}$ constructed just before
this proof.
The result now follows from Theorem~\ref{existence_solution_delay}.
\end{proof}

\subsection{Exponential stability: definitions}
\begin{definition}
\label{def:stab_tel}
Let  $\mathbf{A}:[0,\infty)\to\xR^{2N\times 2N}$ meet Assumption \ref{ass:dissip} (resp. meet Assumption \ref{ass:dissip} and be continuous). For $1\leq p\leq\infty$, 
System (\ref{eq_tel})-(\ref{eq:9}) 
is said to be $L^p$ (resp.\ $C^0$) exponentially stable 
if and only if there exist  $\gamma,K>0$ such that, for all solutions 
given by Theorem~\ref{existence_sol_tele} part I
(resp.\ part II), one has, for all $t \geq 0$,
\begin{equation}
\label{stabilite_eq_telegraph}
\begin{split}
  \left\|\bigl(\widehat{\imath}(t,\cdot),\widehat{v}(t,\cdot)\bigr)\right\|_{L^p([0,1],\xR^{2N})} \leq K e^{-\gamma t} \left\|\bigl(\widehat{\imath}(0,\cdot),\widehat{v}(0,\cdot)\bigr)\right\|_{L^p([0,1],\xR^{2N})}  \hspace{-1em}&    
  \\
  (\,\text{resp.}\  \left\|\bigl(i(t,\cdot),v(t,\cdot)\bigr)\right\|_{C^0([0,1],\xR^{2N})} \leq K e^{-\gamma t}
  \bigl\|\bigl(i(0,\cdot),v(0,\cdot)\bigr)\bigr.&\bigl.\bigr\|_{C^0([0,1],\xR^{2N})}      
  \,).
\end{split}
\end{equation}
\end{definition}
\begin{definition}
\label{def:stab_delay}
Let the maps $t\mapsto D_i(t)$ meet assumption \eqref{ass:borne}
(resp. meet assumption \eqref{ass:borne} and be continuous).
System (\ref{syst_delay_generique})
is said to be $L^p$ (resp.\ $C^0$)  exponentially stable, $p\in[1,\infty]$, 
if and only if there exist $\gamma,K>0$ such that, for all 
solutions given by part \textup{(\textit{i})}
(resp.\ part \textup{(\textit{ii})}) of Theorem \ref{existence_solution_delay}, one has, for all $t \geq 0$,
\begin{equation}
  \label{stabilite_sys_delay_equation}
  \begin{split}
\left\|z(t+\cdot)\right\|_{L^p([-\tau_N ,0],\xR^d)} &\leq K e^{-\gamma t} \left\|z(\cdot)\right\|_{L^p([-\tau_N ,0],\xR^d)} 
  \\
  &\hspace{-6em}(\text{resp.}\; \left\|z(t+\cdot)\right\|_{C^0([-\tau_N ,0],\xR^d)}
  \leq
  K e^{-\gamma t} \left\|z(\cdot)\right\|_{C^0([-\tau_N,0],\xR^d)}
  \,).
  \end{split}
\end{equation}
\end{definition}
There is a slight abuse of notation in
\eqref{stabilite_sys_delay_equation}: for $t\geq0$,
the norms should not apply to $z(t+\cdot):\,[-\tau_N\!-t,+\infty)\to\xR^d$, but rather to its restriction
to $[-\tau_N,0]$.

Our main concern in this paper is the exponential stability of
system (\ref{eq_tel})-(\ref{eq:9}), but we
 shall need
the equivalent formulation 
as a difference-delay system of the form \eqref{eq:10}, which is a particular case of
\eqref{syst_delay_generique}.
In fact, exponential stability of
the two systems are equivalent properties, as asserted by the following proposition.
\begin{proposition}
  \label{prop:equiv}
System (\ref{eq_tel})-(\ref{eq:9}) is $L^p$ exponentially stable (resp.\ $C^0$ exponentially stable)
if and only if
System \eqref{eq:10} is $L^p$ exponentially stable (resp.\ $C^0$ exponentially stable), $1\leq p\leq\infty$.
\end{proposition}
\begin{proof}
  This follows at once from \eqref{eq:11} expressing solutions of
  (\ref{eq_tel})-(\ref{eq:9}) from solutions of \eqref{eq:10} and vice-versa.
\end{proof}

\section{Results}
\label{sec:results}

\subsection{Known results in the time-invariant case}

The exponential stability of difference-delay systems like \eqref{syst_delay_generique} 
when the $D_i$ are constant matrices is well
understood. Indeed, the following necessary and sufficient condition is 
classical.

\begin{theorem}[Henry-Hale Theorem, \cite{Henry1974,Hale}]
  \label{theorem_Hale}
If the matrices $D_i$ in system~\eqref{syst_delay_generique} do not depend on 
$t$, the following properties are equivalent.
\begin{itemize}
  \item[\textup{(\textit{i})}] System (\ref{syst_delay_generique}) is $L^p$ exponentially stable
for all $p\in[1,+\infty]$.
  \item[\textup{(\textit{ii})}] System (\ref{syst_delay_generique}) is $C^0$ exponentially stable.
  \item[\textup{(\textit{iii})}] There exists $\beta<0$ for which
    \begin{equation}
      \label{eq:8}
      Id-\sum\limits_{i=1}^ND_i\,e^{-\lambda \tau_i}\ \text{is invertible for all }
      \lambda\in\xC \text{ such that } \Re(\lambda)>\beta. 
    \end{equation}
\end{itemize}
\end{theorem}
Theorem~\ref{theorem_Hale} is usually stated for $C^0$ exponential stability
only. However, the proof 
yields $L^p$ exponential 
stability as well for $1\leq p\leq\infty$, see 
the discussion after \cite[eq. (1.11)]{CoNg}. 
To study the stability of time-invariant
networks of 1-dimensional hyperbolic systems, 
it is standard to convert them into
difference-delay systems,
much like we did in the previous section, and to
apply Theorem \ref{theorem_Hale}.
There is a sizeable literature on this topic, dealing  with more 
general equations with conservation laws than telegrapher's ones, 
see for instance the textbook \cite{bastin2016stability} and
references therein.

For systems of the form \eqref{eq:10}, if we assume on top of  
the dissipativity
condition \eqref{eq:dissip} that the coupling matrix $\mathbf{A}(t)$ in fact does not depend on 
$t$, then Theorem \ref{theorem_Hale}  applies to yield
 exponential stability. This is the content of the following proposition, 
whose (elementary)
proof is given in section~\ref{sec:constant} for completeness:
\begin{proposition}
\label{prop_cas_constant}
If the matrix $\mathbf{A}(t)$ is constant and condition \eqref{eq:dissip} holds,
then the constant matrices $D_i$ obtained 
when  putting \eqref{eq:10} 
into the form \eqref{syst_delay_generique} satisfy
\eqref{eq:8} for some $\beta<0$.
\end{proposition}

\subsection{Sufficient stability condition in the time-varying case}
\label{sec:results-nous}

There is unfortunately no straightforward generalization of the Henry-Hale theorem to time-varying
difference-delay systems of the form \eqref{syst_delay_generique},
even if we assume that the $D_i(t)$ are periodic with the same period, as is the case in the application 
to electrical networks  outlined  in the introduction.
To the best of our knowledge, there are very few results on the
stability of such  systems; 
let us mention two.  
One is \cite[Lemma~3.2]{CoNg}. It gives exponential stability results in Sobolev norms
for the class of time-varying difference-delay systems \eqref{syst_delay_generique}
which come from 1-D hyperbolic equations, where the matrices $D_i(\cdot)$ are
continuously differentiable and the delays may be time-dependent. 
Another, extensive reference is \cite{Chitour2016}, which  gives
a necessary and sufficient condition for $L^p$ exponential stability 
when  $1\leq p\leq+\infty$ 
that obviously remains  valid for $C^0$ exponential stability as well.
It is stated  in terms of the boundedness of sums of products of the $D_i(t_j)$ at delayed
time intants $t_j$, where the number of terms in the sums and products can be  arbitrary large. This is akin to an expression of the 
solution to
\eqref{syst_delay_generique} in terms of the matrices $D_i(.)$ and the
initial conditions (see \eqref{eq:250} and \eqref{expfMq} further below), which looks difficult to bound efficiently in practice
because of the tremendous
combinatorics and the many cancellations that can occur.
In contrast, we  only deal here with telegrapher's equations,
or with difference-delay systems that can be recast as such,
but Assumption \ref{ass:dissip} is a much more manageable sufficient
 condition for exponential stability.


The main result of the paper ---see Theorems \ref{theorem-central} and \ref{theorem-central-delay} below--- asserts  $L^p$ exponential stability 
for all $p\in[1,\infty]$, as well as $C^0$ 
exponential stability,  for
networks of telegrapher's equations with time-varying coupling conditions of the form
or \eqref{eq_tel}-\eqref{eq:9}
(or \eqref{eq_tel}-\eqref{equ_Boundary}) under 
Assumption \ref{ass:dissip} (dissipativity at the nodes),
and for difference-delay systems \eqref{syst_delay_generique} under conditions that imply that they
can be put in the form \eqref{eq:10} with the same dissipativity conditions.

It may be interesting to note that the sufficient condition for
  stability that we give here is independent of the delays when
  speaking of a difference-delay system (Theorem
  \ref{theorem-central-delay}) or independent of the characteristics of
  the lines (constants $C_k$ and $L_k$) when speaking of networks of
  telegrapher's equations (Theorem \ref{theorem-central}).
  Also, these sufficient conditions are not claimed to be necessary.

Let us state these results, preceded by some auxiliary  results of independent interest.
 The proofs not given right after the theorems can be found in
Sections~\ref{sec:proof-gen} through \ref{sec:proof-thDelay}. 

The first step is to establish $L^2$ exponential stability of System (\ref{eq_tel})-(\ref{eq:9}) asserted in
the following theorem.
We give in Section~\ref{sec:proof-gen} a proof using a natural energy functional as Lyapunov function for
the telegrapher equations \eqref{eq_tel}.
Condition \eqref{eq:dissip} in Assumption \ref{ass:dissip}, which has been termed 
dissipativity without much
explanation so far, expresses dissipativity in the sense of
this energy functional.
We also sketch, in Section~\ref{sec:proof-gen_TDS}, a second proof,
elaborating on \cite[Lemma 3.2]{CoNg}, which is exclusively based on the time-varying difference-delay
system~\eqref{eq:10}; see the remark at the end of
Section~\ref{sec:proof-gen_TDS} for a comparison of the two proofs.
\begin{theorem}
\label{theorem_L2_stability}
Under Assumption \textup{\ref{ass:dissip}}, 
the time-varying network of telegrapher's equations \eqref{eq_tel}-\eqref{eq:9}-\eqref{eq:7}  is $L^2$ exponentially stable.
\end{theorem}
\noindent
In view of Proposition~\ref{prop:equiv}, we get as a corollary that
$L^2$ exponential stability holds for 
difference-delay systems of the form \eqref{eq:10}.
\begin{rmrk}[On other sufficient conditions than Assumption \ref{ass:dissip}]
  \label{rmk-AutreCS}
  Another more general result could have been stated here, for
  arbitrary $q$ in $[1,+\infty]$, concluding $L^q$ exponentially
  stability from the assumption that
  \\- the maps $t\mapsto Id+ \mathbf{A}(t)\mathbf{K}$ and
  $t\mapsto (Id+ \mathbf{A}(t)\mathbf{K})^{-1}$ are measurable and
  bounded,
  \\- there exists a number $\nu$, $0<\nu<1$ and an invertible
  diagonal matrix $D$ such that
  \begin{equation}
    \label{eq:14}
    \vertiii{D\left(Id+ \mathbf{A}(t)\mathbf{K}\right)^{-1}
    \left(Id-\mathbf{A}(t)\mathbf{K}\right)P_2\,D^{-1}}_q<\nu
  \ \text{for (almost) all }t,
  \end{equation}
  where $\vertiii{\cdot}_q$ denotes the operator norm 
  with respect to the $q$ norm in $\xR^{2N}$.
  See Remark~\ref{rmk-proofAutreCS} in
  section~\ref{sec:proof-gen_TDS} for a proof, in the spirit of \cite{CoNg}.
  Also, in view of Theorem~\ref{th:L2impliesC0} below, the same
  conditions 
  would also imply $L^p$ exponential
  stability for any $p\in[1,\infty]$, and $C^0$ stability if the map
  $\mathbf{A}(\cdot)$ is continuous, yielding a version of
  Theorem~\ref{theorem-central} below with Assumption~\ref{ass:dissip}
  replaced by the above conditions for one $q$ in $[1,\infty]$.

  If $q=2$, the above conditions are implied, with $D=\mathbf{K}^{1/2}$, 
  by Assumption~\ref{ass:dissip} (see \eqref{eq:Proof_TDS1} in
  Section~\ref{sec:proof-gen_TDS} and Remark~\ref{rmk-hypWellPosed}
  for well-posedness), but are clearly more general.
  Hence, the sketched results are indeed formally more general than
  Theorems~\ref{theorem_L2_stability} and
  \ref{theorem-central}.
  We however chose to keep them as a remark and
  to keep Assumption~\ref{ass:dissip} in the exposition because it is
  much more explicit and is the natural dissipativity assumption in the
  applications described in the introduction.
%
%
\end{rmrk}

To deduce $L^p$ exponential stability, for all $p$, from Theorem \ref{theorem_L2_stability}, we  rely on the following result.
\begin{theorem}
  \label{th:L2impliesC0}
Under Assumption \textup{\ref{ass:borne}}, System \eqref{syst_delay_generique} is 
$L^p$ exponentially stable for some $p\in[1,\infty]$ if and only if it is 
$L^p$ exponentially stable for all such $p$. Moreover, if the maps
$t\mapsto D_i(t)$ are continuous, then this is also if and only if
System \eqref{syst_delay_generique} is $C^0$ exponentially stable.
\end{theorem}
\noindent
The only original bit here
is that $C^0$ exponential stability implies $L^p$ exponential stability for 
all $p$, because the first assertion of Theorem \ref{th:L2impliesC0}
is essentially contained in \cite[Corollary  3.29]{Chitour2016}.
We do consider $C^0$ stability, because it is the natural one in
the application to  electronic circuits mentioned
in the introduction.
Although, again, the first assertion 
is a consequence of \cite[Cor.\ 3.29]{Chitour2016},
we nevertheless give an independent proof in Section~\ref{sec:proof-L2C0}.
Indeed, we feel  our argument is simpler than
in \cite{Chitour2016} (the latter paper contains of course other results), 
and of independent interest. 
Moreover,
our proof  shows (for better readability it is not stated in the theorem) that if System \eqref{syst_delay_generique}
is $L^p$ (resp. $C^0$) polynomially stable of degree $m>N$ for some $p\in[1,\infty]$
({\it i.e.}  if 
\eqref{stabilite_sys_delay_equation} 
holds with $e^{-\gamma t}$ replaced by 
$(1+t)^{-m}$), then it is $L^p$ polynomially stable of degree 1 for all such $p$ (and also $C^0$ polynomially stable of degree $1$).

An obvious corollary of Theorem \ref{th:L2impliesC0}, based on Proposition~\ref{prop:equiv}, is that System
(\ref{eq_tel})-(\ref{eq:9}) (network of telegrapher's equations) is $L^p$ exponentially stable for
some $p\in[1,\infty]$ if and only if it is $C^0$ exponentially stable 
and also $L^q$ exponentially stable for all $q\in[1,\infty]$.
This leads to our main result regarding network of telegrapher's equations:
\begin{theorem}
\label{theorem-central}
Under Assumption~\textup{\ref{ass:dissip}}, the time-varying network of telegrapher's equations \eqref{eq_tel}-\eqref{eq:9}-\eqref{eq:7} is $L^p$ exponentially 
stable for $1\leq p\leq\infty$, and also $C^0$ exponentially stable if the maps
$t\mapsto\mathbf{A}(t)$ are continuous.
\end{theorem}
\begin{proof}
  This is a straightforward consequence of Theorem~\ref{theorem_L2_stability} and the ``obvious
  corollary'' to Theorem~\ref{th:L2impliesC0} mentioned just before Theorem~\ref{theorem-central}.
\end{proof}


A direct consequence of Theorem \ref{theorem-central} and 
Proposition~\ref{prop:equiv} is that the same stability properties
hold for difference-delay
systems of the special form \eqref{eq:10}.
It is interesting to restate this  in terms
of general delay systems of the form \eqref{syst_delay_generique}, making additional assumptions to fall under the scope of the previous result.
This is the purpose of Theorem~\ref{theorem-central-delay} below, whose proof is given in Section~\ref{sec:proof-thDelay}.
Recall that the matrices $\widehat{D}_j(t)$ were defined from the matrices $D_i(t)$
in Remark~\ref{rmk-repeated-delays} ({\it cf.} \eqref{eq:21}),
and that they differ from the
$D_i$ only when some of the delays $\eta_i$ appear with repetition  in
\eqref{syst_delay_generique}).
\begin{theorem}
  \label{theorem-central-delay}
  Under Assumption \textup{\ref{ass:borne}}, if Conditions \textup{(\textit{i})} and \textup{(\textit{ii})} below are satisfied, then
  the time-varying difference-delay system \eqref{syst_delay_generique} 
  is $L^p$ exponentially stable for all $p\in[1,\infty]$. 
  Moreover, if the maps $t\mapsto D_i(t)$ are continuous, then it is also $C^0$ exponentially stable.
  \begin{itemize}
  \item[\textup{(\textit{i})}] The columns of the matrices $\widehat{D}_j(t)$ are disjoint, i.e. there is a partition
    $\{1,\ldots,d\}=\mathcal{I}_1\cup\cdots\cup\mathcal{I}_{\widehat{\newN}}$
    (with $i\!\neq\!j\Rightarrow\mathcal{I}_i\cap\mathcal{I}_j\!=\!\varnothing$)
    such that the
    $k$\textsuperscript{th} column of $\widehat{D}_j(t)$ is identically zero if $k\notin \mathcal{I}_j$.
  \item[\textup{(\textit{ii})}] The sum of the matrices $D_i(t)$ is uniformly contractive, \emph{i.e.}
    there is a number $\nu<1$ such that
    $\displaystyle\vertiii{\sum_{i=1}^\newN D_i(t)}\leq\nu$ for almost all positive $t$\,.
  \end{itemize}
\end{theorem}
\noindent
Here, $\vertiii{\cdot}$ is the spectral norm for matrices associated to the Euclidean norm on
$\xR^d$, like in section~\ref{sec-existence}.

\medskip

To recap, Theorem~\ref{theorem-central} offers a sufficient condition for exponential stability of
networks of coupled telegrapher's equations,
relevant to the study of oscillations in circuits with transmission lines as explained in Section \ref{sec:intro},
while Theorem~\ref{theorem-central-delay} deals with difference-delay systems
and  applies to an admittedly narrow class thereof
(the disjoint columns assumption is clearly restrictive),
but is still worth stating for it points at a class of time-varying systems for which relatively
simple sufficient conditions for exponential stability can be given.
These results are apparently first to give fairly manageable sufficient conditions 
for exponential stability in the time-varying case. 
Another contribution is the somewhat simpler approach,
provided by Theorem~\ref{th:L2impliesC0} and its proof,
to the fact that all types of $L^p$ exponential stability, $1\leq p\leq\infty$, are equivalent for
general time-varying difference-delay systems.

\section{Proofs}
\label{sec:proofs}

\subsection{A technical lemma}
\label{sec:lemmas}

Here, the superscript $*$ denotes the transpose of a real matrix, and the spectral norm $\vertiii{\cdot}$ defined at the
beginning of section~\ref{sec-existence} is with respect to the canonical Euclidean norm $\|x\|=(x^*x)^{1/2}$.
\begin{lemma}
  \label{lemma1}
  If $Q$ is a square matrix satisfying $Q+Q^*>\kappa \,Id>0$, there is a unique square matrix $R$ solution of
  \begin{equation}
    \label{eq:5}
    (Id+Q)R=Id-Q\,,
  \end{equation}
  and it satisfies $\vertiii{R}<(1-\kappa)/(1+\kappa)<1$. Conversely,
  if $R$ is a square matrix satisfying $\vertiii{R}<1$, there is a unique square matrix $Q$ solution of
  \eqref{eq:5} and it satisfies
  $Q+Q^* \geq \frac{1-\vertiii{R}}{1+\vertiii{R}}\,Id\,.$
\end{lemma}
\begin{proof}
  It is clear that $-1$ cannot be an eigenvalue of $Q$ if $Q+Q^*>0$ or an eigenvalue of $R$ if
  $\vertiii{R}<1$. This allows to solve for $R$ or $Q$ using the inverse of $Id+Q$ or $Id+R$.
  
  Now suppose that $R$ and $Q$ satisfy \eqref{eq:5}. Then $(Id+Q)(Id+R)=2\,Id$, hence both $Id+Q$ and $Id+R$ are
  invertible and $R$ commutes with $Q$, \eqref{eq:5} can be re-written $R(Id+Q)=Id-Q$ that readily implies
\begin{displaymath}
\frac{\|R\,(Id+Q)y\|^2}{\|(Id+Q)y\|^2}
=
1-2\,\frac{y^*(Q+Q^*)y}{\|(Id+Q)y\|^2}
\end{displaymath}
  for any nonzero $y$. Since on the one hand, using invertibility of $Id+Q$, $\vertiii{R}<1$ if and only
  if the left-hand side is less than 1 for any nonzero $y$ and on the other hand the right-hand side is less than one
  if and only if $y^*(Q+Q^*)y$ is positive, one deduces that $Q+Q^*>0$ and $\vertiii{R}<1$ are equivalent.
\end{proof}

\subsection{Proof of Proposition~\ref{prop_cas_constant}}
\label{sec:constant}

From the very definition of $D_i$ in terms of  $\mathbf{A}$, $\mathbf{K}$  and $P_2$
---see discussion after \eqref{eq:10}--- we get that
\begin{equation}
    \label{eq:19}
    \sum_{i=1}^ND_i\,e^{-\lambda \tau_i} =
    (Id\!+\! \mathbf{A}\mathbf{K})^{-1} (Id\!-\!\mathbf{A}\mathbf{K})P_2 
    \diag(e^{-\lambda \tau_1}\ldots e^{-\lambda \tau_N}, e^{-\lambda \tau_1}\ldots e^{-\lambda \tau_N})\,.
\end{equation}
In view of \eqref{eq:dissip}, \eqref{eq:9.5} and the strict positivity of 
the $K_j$,
it holds if we set $Q=\mathbf{K}^{1/2}\mathbf{A}
\mathbf{K}^{1/2}$ that $Q+Q^*\geq\tilde{\alpha}Id$ with
$\tilde{\alpha}=\alpha \min_{1\leq j\leq N}K_j>0$,
hence Lemma \ref{lemma1} gives us:
\begin{eqnarray}
\label{ine_to_prove}
\vertiii{(Id+\mathbf{K}^{1/2}\mathbf{A} \mathbf{K}^{1/2})^{-1}(Id-\mathbf{K}^{1/2}\mathbf{A} \mathbf{K}^{1/2})} <1. 
\end{eqnarray}
  
Consider now the $\mathbf{K}$-norm on $\xR^{2N}$, defined 
for $x \in \xR^{2N}$ by $\|x\|_{\mathbf{K}}=\|\mathbf{K}^{\frac{1}{2}}x\|$.
Clearly, for any $2N\times 2N$ complex matrix $B$, the corresponding operator norm
is $\vertiii{B}_{\mathbf{K}}=\vertiii{\mathbf{K}^{\frac{1}{2}} B \mathbf{K}^{-\frac{1}{2}}}$;
it is obviously multiplicative.

Since
$[Id+ \mathbf{A}\mathbf{K}]^{-1} [Id-\mathbf{A}\mathbf{K}] = \mathbf{K}^{-\frac{1}{2}}
(Id+\mathbf{K}^{\frac{1}{2}}\mathbf{A}\mathbf{K}^{\frac{1}{2}})^{-1}
(Id-\mathbf{K}^{\frac{1}{2}}\mathbf{A}\mathbf{K}^{\frac{1}{2}}) \mathbf{K}^{\frac{1}{2}}$,
equation \eqref{ine_to_prove} implies that
\begin{eqnarray}
  \label{eq:20-0}
\vertiii{[Id+ \mathbf{A}\mathbf{K}]^{-1} [Id-\mathbf{A}\mathbf{K}]}_{\mathbf{K}}<1,
\end{eqnarray}
consequently there is  $\beta<0$ such that
\begin{equation}
  \label{eq:20}
  \vertiii{[Id+\mathbf{A}\mathbf{K}]^{-1} [Id-\mathbf{A}\mathbf{K}]}_{\mathbf{K}} \; e^{-\beta\,\tau_N}\,<1\,.
\end{equation}
To see that \eqref{eq:8} holds for this $\beta$, pick
$\lambda\in\mathbb{C}$ with  $\Re(\lambda)>\beta$ and observe that
\begin{displaymath}
  \vertiii{P_2 \diag(e^{-\lambda \tau_1},\ldots, e^{-\lambda \tau_N}, e^{-\lambda \tau_1},\ldots, e^{-\lambda
    \tau_N})}_\mathbf{K}\leq e^{-\beta\tau_N}\,
\end{displaymath}
by \eqref{eq:1}, the multiplicativity of the $\mathbf{K}$-norm and the fact that $P_2$ commutes with $\mathbf{K}^{1/2}$. Hence, using \eqref{eq:19} and
\eqref{eq:20} together with the  multiplicativity of the $\mathbf{K}$-norm,  
we see that $\vertiii{\sum\limits_{i=1}^ND_i\,e^{-\lambda \tau_i}}_{\mathbf{K}}<1$
which implies \eqref{eq:8}.
\hfill$\square$

\subsection{Proof of Theorem~\ref{theorem_L2_stability} \textit{via} a Lyapunov functional approach  }
\label{sec:proof-gen}

Let $(v_1(t,x),\ldots,v_N(t,x),i_1(t,x),\ldots,i_N(t,x))\in L^1_{loc}([0,\infty),(L^2([0,1]))^{2N})$ be the solution to \eqref{eq_tel}-\eqref{eq:9}-\eqref{eq:7} set forth in Part I of 
Theorem \ref{existence_sol_tele}, with initial condition $\InitCond{i}_k,\InitCond{v}_k\in L^2([0,1])$ for $1\leq k\leq N$.
We define the energy functional $E_k$ in the line $k$ and the global energy $E$ by
\begin{equation}
\label{energy}
  E_k(t)=\frac{1}{2}\int_0^1 \left[C_kv_k^2(t,x)+L_ki_k^2(t,x)\right]\mathrm{d}x\,,
  \ \ \ \ 
  E(t)=\sum_{k=1}^N E_k(t)\,.
\end{equation}


\textit{Fact.} 
Each function $t\mapsto E_k(t)$ is
locally absolutely continuous and its derivative
is given by:
\begin{eqnarray}
\label{derEkc}
\frac{d}{dt}\,E_k(t)=-\widehat{\imath}_k(t,1)\widehat{v}_k(t,1)+\widehat{\imath}_k(t,0)\widehat{v}_k(t,0),\qquad \text{a.e.} \ t.
\end{eqnarray}

\textit{Proof of the Fact.}
This would be easy if the solution was smooth (differentiating under the integral sign and using
\eqref{eq_tel} would readily yield \eqref{derEkc}), but
we have only proved so far, according to
Remark \ref{remc}, that $E_k$ is continuous $[0,+\infty)\to\xR$ for
each $k$.
In particular it defines a distribution on $(0,+\infty)$; let us compute the derivative of this
distribution by approximation.
By Proposition \ref{prop:resolu}, points  (\textit{i})-(\textit{ii}),
the functions $v_k$, $i_k$ are of the form  \eqref{eq:0}
with $f_k\in L^2_{loc}((-\infty,1])$ and
$g_k\in L^2_{loc}([0,\infty))$. Let  $\check{f}_k$ and $\check{g}_k$
extend $f_k$ and $g_k$ by zero to the whole of $\xR$,  and pick
$\phi:\xR\to\xR$ a positive, 
$C^\infty$-smooth  function, supported on $[-1,1]$ and  such that
$\int_{\xR}\phi=1$. For each $\varepsilon>0$, we set
$\phi_\varepsilon(x):=\phi(x/\varepsilon)/\varepsilon$ (hence, $\int_{\xR}\phi_\varepsilon=1$) and define
\begin{equation}
\label{deftfg}
\widetilde{f}_{k,\varepsilon}(s):=\int_{\xR}\check{f}(y)\phi_\varepsilon
(s-y)\,\mathrm{d}y,\qquad
\widetilde {g}_{k,\varepsilon}(s):=\int_{\xR}\check{g}(y)\phi_\varepsilon(s-y)\,\mathrm{d}y\,.  
\end{equation}
Thus,
$\widetilde{f}_{k,\varepsilon}\in L^2_{loc}(\xR)$ is
$C^\infty$ smooth and satisfies
$\|\widetilde{f}_{k,\varepsilon}\|_{L^2(K)}\leq \|f_k\|_{L^2(K+[-\varepsilon,\varepsilon])}$ for any compact $K\subset\xR$, and similarly for 
$\widetilde{g}_{k,\varepsilon}$. 
Moreover, $\widetilde{f}_{k,\varepsilon}$ and
$\widetilde{g}_{k,\varepsilon}$
converge, both pointwise a.e. and in $L^2_{loc}(\xR)$, respectively 
to $\check{f}_k$ and $\check{g}_k$, when $\varepsilon\to0$. Indeed, it is enough to check this on an arbitrary compact set 
$K\subset\xR$,
and since $\phi_\varepsilon$ is supported on $[-\varepsilon,\varepsilon]$
we may redefine $\check{f}_k$ and $\check{g}_k$
as being zero outside the compact set $K+[-\varepsilon,\varepsilon]$
without changing the values of $\widetilde{f}_{k,\varepsilon}$ nor
$\widetilde{g}_{k,\varepsilon}$ on $K$. Thus, it is enough to prove
the desired pointwise and $L^2_{loc}$ convergence when 
$\check{f}_k$ and $\check{g}_k$ lie in $L^2(\xR)$, in which case
the result is standard \cite[ch. III, thm. 2]{Stein}. Next,
let us put
\begin{equation}
\label{deftid}
\begin{split}
  &\widetilde{v}_{k,\varepsilon}(t,x):=\widetilde{f}_{k,\varepsilon}
  (x-\frac{t}{\tau_k})+\widetilde{g}_{k,\varepsilon}(x+\frac{t}{\tau_k})
  \,, 
  \\&\widetilde{\imath}_{k,\varepsilon}(t,x)=K_k\left(\widetilde{f}_{k,\varepsilon}(x-\frac{t}{\tau_k})-\widetilde{g}_{k,\varepsilon}(x+\frac{t}{\tau_k})\right),
\end{split}
\end{equation}
so that $\widetilde{v}_{k,\varepsilon}$ and  $\widetilde{\imath}_{k,\varepsilon}$ 
lie in $L^2_{loc}(\xR^2)$ and are $C^\infty$ smooth solutions to \eqref{eq_tel} on $\xR^2$, by Proposition~\ref{prop:resolu} point (\textit{iii}). Because $(t,x)\mapsto(x-t/\tau_k,x+t/\tau_k)$ is a bi-Lipschitz homeomorphism of $\xR^2$, it preserves compact sets and sets of measure zero. Thus, since 
$\check{f}_k$ and $\check{g}_k$ 
coincide respectively with $f_k$ and $g_k$ on $[0,+\infty)\times[0,1]$,
the properties of 
$\widetilde{f}_{k,\varepsilon}$ and $\widetilde{g}_{k,\varepsilon}$ 
indicated after \eqref{deftfg} imply that $\widetilde{v}_{k,\varepsilon}$,
$\widetilde{\imath}_{k,\varepsilon}$ respectively converge 
pointwise a.e. to $v_k$, $i_k$ on $[0,+\infty)\times[0,1]$, in such a way that
$\|\widetilde{v}_{k,\varepsilon}(t,.)\|_{L^2([0,1])}$ and
$\|\widetilde{\imath}_{k,\varepsilon}(t,.)\|_{L^2([0,1])}$ remain essentially
bounded with $t$. Therefore, 
by dominated convergence, we get for every $C^\infty$ smooth compactly supported
function
$\psi:(0,+\infty)\to\xR$ that
\begin{multline*}
  \lim_{\varepsilon\to0} \;
\int_0^{+\infty}\!\!\!\int_0^1
\Bigl(C_k\widetilde{v}_{k,\varepsilon}^2(t,x)+L_k\widetilde{\imath}_{k,\varepsilon}^2(t,x)\Bigr)\;\psi(t)\,\mathrm{d}t\,\mathrm{d}x\\=
\int_0^{+\infty}\!\!\!\int_0^1 \Bigl(C_kv_k^2(t,x)+L_ki_k^2(t,x) \Bigr)\,\psi(t) \;\mathrm{d}t\,\mathrm{d}x.
\end{multline*}
In other words: when $\varepsilon\to0$, then
$\widetilde{E}_{k,\varepsilon}(t):= \int_0^1 \left(C_k\widetilde{v}_{k,\varepsilon}^2(t,x)+L_k\widetilde{\imath}_{k,\varepsilon}^2(t,x)\right)\mathrm{d}x$ converges to $E_k(t)$,
as a distribution on $(0,+\infty)$. 
Now, since $\widetilde{\imath}_{k,\varepsilon}$ and $\widetilde{v}_{k,\varepsilon}$
are smooth, the derivative of $t\mapsto\widetilde{E}_{k,\varepsilon}(t)$ can be 
computed in the strong sense by differentiating under the integral sign; since
$\widetilde{\imath}_{k,\varepsilon}$ and $\widetilde{v}_{k,\varepsilon}$ are solutions of the telegrapher's equation~\eqref{eq_tel}, 
an elementary integration yields:
\begin{eqnarray}
\label{dertid}
\frac{d}{dt}\widetilde{E}_{k,\varepsilon}(t)=-\widetilde{\imath}_{k,\varepsilon}(t,1)
\widetilde{v}_{k,\varepsilon}(t,1)+\widetilde{\imath}_{k,\varepsilon}(t,0)\widetilde{v}_{k,\varepsilon}(t,0).
\end{eqnarray}
By \eqref{deftid} and the Schwarz inequality, 
the properties of $\widetilde{f}_{k,\varepsilon}$ and $\widetilde{g}_{k,\varepsilon}$ 
indicated after \eqref{deftfg} imply that
the right hand side of \eqref{dertid} converges pointwise a.e. and
in $L^1_{loc}(\xR)$ to the function
\[
F(t):=K_k\left(f_k^2(-\frac{t}{\tau_k})-g_k^2(\frac{t}{\tau_k})
-f_k^2(1-\frac{t}{\tau_k})+g_k^2(1+\frac{t}{\tau_k})\right),
\]
and since we know that $\frac{d}{dt}\widetilde{E}_{k,\varepsilon}$ converges to $\frac{d}{dt} E_k$ 
as a distribution we conclude that $\frac{d}{dt} E_k=F$. In particular, 
since $E_k$ is a distribution in dimension 1 whose
derivative is a locally integrable function, \cite[thm. 6.74]{Demengel}
implies local absolute continuity
and we get from what precedes that
$\frac{d}{dt} E_k(t)=F(t)$ for a.e. $t$, which can be rewritten as
\eqref{derEkc} in view  of \eqref{eq:15a} and \eqref{eq:15b}.  This ends the proof of the above fact.

\begin{proof}[Proof of Theorem~\ref{theorem_L2_stability}]
Adding equalities \eqref{derEkc} for $1\leq k\leq N$ and considering \eqref{energy} together with
the boundary conditions \eqref{eq:6} yields the following equation,
where one may indifferently use
$\frac12 \left( \mathbf{A}(t) + \mathbf{A}(t)^*\right)$ or $\mathbf{A}(t)$:
\begin{eqnarray}\label{eq:Edot}
\frac{d}{dt}\,E(t)=-
\begin{pmatrix} -\widehat{\imath}_1(t,0)\\\vdots\\-\widehat{\imath}_N(t,0)\\\widehat{\imath}_1(t,1)\\\vdots\\\widehat{\imath}_N(t,1)  \end{pmatrix}^*
\frac{\mathbf{A}(t) + \mathbf{A}(t)^*}{2}
\begin{pmatrix}
  -\widehat{\imath}_1(t,0)\\\vdots\\-\widehat{\imath}_N(t,0)\\\widehat{\imath}_1(t,1)\\\vdots\\\widehat{\imath}_N(t,1)  \end{pmatrix},
  \qquad \text{a.e.} \ t,
\end{eqnarray}
Using the dissipativity condition \eqref{eq:dissip} in \eqref{eq:Edot}
readily implies:
\begin{eqnarray}
\label{eq:121}
\frac{d}{dt}\,E(t) \leq -\frac\alpha2 \sum\limits_{k=1}^{N} \left[ \widehat{\imath}^2_k(t,0)+ \widehat{\imath}^2
_k(t,1)\right],\qquad \text{a.e.} \  t.
\end{eqnarray}
This entails  that the global energy $E$ is
decreasing. In order to show that  it tends to zero exponentially, let us express $E$
in terms of the functions $f_k$, $g_k$ as follows.
Substituting \eqref{eq:0} in \eqref{energy}, we get since $L_k K_k^2=C_k$ that
\begin{eqnarray}
  \label{eq:Eaternatif}
E_k(t)=C_k\left(\int_0^1g_k^2(x+\frac{t}{\tau_k})\mathrm{d}x+
\int_0^1f_k^2(x-\frac{t}{\tau_k}) \mathrm{d}x\right).
\end{eqnarray}
Changing variables to $\tau=x\tau_k+t$ in the first
integral and to $\tau=(1-x)\tau_k+t$ in the second, we obtain:
\begin{eqnarray}
\label{expEkfg}
E_k(t)= C_k \int_t^{t+\tau_k}\left(g_k^2(\frac{\tau}{\tau_k})+f_k^2(1-\frac{\tau}{\tau_k})\right)d\tau.
\end{eqnarray} 
Thus, if we define $G: (0,+\infty)\to\xR$ by
$G(\tau):=\sum_{k=1}^N C_k \bigl(g_k^2(\tau/\tau_k)+f_k^2(1-\tau/\tau_k)\bigr)$, 
we deduce from \eqref{expEkfg} that
\begin{eqnarray}
\label{eqdemo1}
E(t) \leq\int_t^{t+\tau_N} G(\tau)\,d\tau.
\end{eqnarray}
In another connection, we get from \eqref{eq:0} that 
$G(\tau)$ can be expressed as a non-negative quadratic form in  the $4N$ variables
$\widehat{v}_k(\tau,0)$, $\widehat{v}_k(\tau,1)$,
$\widehat{\imath}_k(\tau,0)$, $\widehat{\imath}_k(\tau,1)$, for
$1\leq k\leq N$, with
constant coefficients. 
Hence, using \eqref{eq:6} to substitute the $\widehat{v}_k$'s for
the $\widehat{\imath}_k$'s, the same $G(\tau)$ can be expressed as a non-negative quadratic form in  the $2N$ variables
$\widehat{\imath}_k(\tau,0)$, $\widehat{\imath}_k(\tau,1)$, for
$1\leq k\leq N$, with
time-varying essentially bounded coefficients (Assumption \ref{ass:dissip}) depending on the matrix
$\mathbf{A}(\tau)$ and the constants $K_k$. This implies:
\begin{eqnarray}
\label{eq:122}
G(\tau) \;\leq\; \widetilde{a} \;\sum\limits_{k=1}^{N} \left( \,\widehat{\imath}_k ^2 (\tau,0) + \widehat{\imath}_k ^2 (\tau,1) \right),\qquad \text{a.e.}\ \tau>0.
\end{eqnarray}
with a positive constant $\widetilde{a}$ that depends only on
the coefficients $K_k$ and the bounds on the coefficients of $\mathbf{A}(.)$.
Using this inequality in \eqref{eq:121} yields
\begin{equation}
\label{eqdemo2}
\frac{d}{d\tau}\,E(\tau) \leq -\frac{\alpha}{2\,\widetilde{a}}\,G(\tau),\qquad \text{a.e.}\ \tau>0.
\end{equation}
Integrating \eqref{eqdemo2} between $t$ and $t+\tau_N$ we gather,
in view of \eqref{eqdemo1}, that
\begin{eqnarray}
0\leq E(t+\tau_N) \leq (1-\frac{\alpha}{2\,\widetilde{a}})\,E(t),\qquad t>0.
\end{eqnarray}
Comparing the expression of $E(t)$ in \eqref{energy},
this last  inequality readily implies that system \eqref{eq_tel}-\eqref{eq:9} is $L^2$ exponentially
stable.
\end{proof} 

\subsection{Sketch of an alternative proof of
  Theorem~\ref{theorem_L2_stability} \textit{via} difference-delay
  systems exclusively}
\label{sec:proof-gen_TDS}

First note that, from a straightforward generalization of 
  Equation~\eqref{eq:20-0} using the fact that $P_2$ and $\mathbf{K}$
  commute, there is a $\gamma \in (0,1)$ independent of $t \in \mathbb{R}$ such that :
\begin{equation}
\label{eq:Proof_TDS1}
  \vertiii{[Id+\mathbf{A}(t)\mathbf{K}]^{-1} [Id-\mathbf{A}(t)\mathbf{K}]\,P_2}_{\mathbf{K}} \leq \gamma<1.
\end{equation}
By applying $ \|\cdot\|_{\mathbf{K}}^{\,2}$ to each side of
Equation~\eqref{eq:10}, using \eqref{eq:Proof_TDS1} above and
integrating the resulting inequality between $t$ and $t_2$ ($-\tau_N<t<t_2$),
one gets after simple algebraic manipulation the following
inequality, valid for any $t>0$ and $t_2>t$:
\begin{eqnarray}
\label{eq:Proof_TDS3}
\int_{t}^{t_2}  \|z(s)\|_{\mathbf{K}}^2\mathrm{d}s \leq \frac{1}{1-\gamma^2} \int_{t}^{t+\tau_N} \|z(s)\|_{\mathbf{K}}^2\mathrm{d}s\,,
\end{eqnarray}
in which one may then take $t_2=+\infty$.
This implies $L^2$ exponential stability of system~\eqref{eq:10}
(one first proves that, for $T$ large enough,
$\int_{t}^{t+T}\|z(s)\|_{\mathbf{K}}^2\mathrm{d}s$ converges exponentially to zero)
and thus  $L^2$ exponential stability of system
\eqref{eq_tel}-\eqref{eq:9} \textit{via} the equivalence between
stability of the difference-delay system and of the PDE network, see Proposition~\ref{prop:equiv}.

\begin{rmrk}
  The above proof expounds that of \cite[Lemma 3.2]{CoNg}, 
  but in essence
is not so different from the previous one. 
Indeed, the
quantity $\int_{t}^{t+T}\|z(s)\|_{\mathbf{K}}^2\mathrm{d}s$ acts as a
Lyapunov function for \eqref{eq:10}, although it is not proved to be
non-increasing with respect to continuous time, while $E$, that also has an
expression in terms of the difference-delay system
(see \eqref{eq:Eaternatif}), is a Lyapunov function in the
usual sense for the network of telegrapher's equations, see \eqref{eq:121}.
\end{rmrk}
\begin{rmrk}
  \label{rmk-proofAutreCS}
  The ``more general'' result we sketched in Remark~\ref{rmk-AutreCS}
  can be proved as follows. On the one hand, as already noticed in Remark~\ref{rmk-hypWellPosed},  the first point provides
  existence and uniqueness of the solutions. On the other hand,
  instead of deducing \eqref{eq:Proof_TDS1} from previous statements,
  one directly uses \eqref{eq:14} after applying $ \|D\cdot D^{-1}\|_{p}^{\,p}$ to each side of
Equation~\eqref{eq:10}. 
\end{rmrk}

\subsection{Proof of Theorem~\ref{th:L2impliesC0}}
\label{sec:proof-L2C0}
Before proceeding with the proof, we take a closer look at the structure
of solutions to System \eqref{syst_delay_generique}.

Given the ordered collection of delays $0<\tau_1\leq\tau_2\leq\ldots\leq\tau_N$,
we define the following subsets of  $\xR$:
\begin{equation}
    \label{eq:13}
   \Sigma=\{\sum_{i=1}^N q_i \tau_i\,,\;(q_1,\ldots,q_N)\in\xN^N\}\,\ 
    \text{and}\ \ \ \Sigma_t=[0,t]\cap\Sigma\ \ \ \text{for $t$ in $[0,+\infty)$.}
\end{equation}
Call $Q(t)\in\xN$ the cardinality of $\Sigma_t$.
Clearly, $Q(t)$ is no larger than the number of
$N$-tuples $(q_1,\ldots,q_N)\in\xN^N $ satisfying $\sum_{i=1}^N q_i\leq t/\tau_1$, and the latter is bounded from above by $(1+[[t/\tau_1]])^{N}$, where $[[r]]$ indicates the integer part of the
real number $r$. Hence, we have that
\begin{equation}
  \label{eq:249a}
  Q(t)\leq \left(1+\frac{t}{\tau_1}\right)^{N}\,,\qquad\ t\in[0,+\infty).
\end{equation}
We enumerate the elements of $\Sigma$ as a sequence
$0=\sigma_1<\sigma_2<\sigma_3\cdots$, so that $\Sigma_t$ is described as:  
\begin{equation}
  \label{eq:249b}
  \Sigma_t=\{\sigma_1,\sigma_2,\ldots,\sigma_{Q(t)}\} \,,\qquad \ t\in[0,+\infty)\,.
\end{equation}
Our proof of Theorem \ref{th:L2impliesC0} will dwell on the following observation.

\smallskip

\textit{Fact.}
There is a collection of maps $(M_q)_{q\in\xN}$ from $\xR$ into
$\xR^{d\times d}$ enjoying the three properties (\textit{i}), (\textit{ii}), (\textit{iii}) below.
\begin{itemize}
\item[(\textit{i})] The map $M_q$ lies in $L^\infty_{loc}(\xR, \xR^{d\times d})$,
\item[(\textit{ii})] $M_q$ satisfies 
\begin{equation}
  \label{eq:16}
    t\notin (\sigma_q-\tau_N, \sigma_q]\;\Rightarrow\;M_q(t)=0\,,
\end{equation}
\item[(\textit{iii})] the solution $t\mapsto  z(t)$ of \eqref{syst_delay_generique} with
initial condition
$z(t')=\phi(t')$, $t'\in[-\tau_N,0]$, is given by
\begin{eqnarray}\label{eq:250}
  z(t) =  \sum_{q=1}^{Q(t+\tau_N)}
  M_q(t) \, \phi(t-\sigma_q)
  =  \sum_{q=1}^{+\infty}
  M_q(t) \, \phi(t-\sigma_q),\qquad t\geq0. 
\end{eqnarray}
\end{itemize}
Formula \eqref{eq:250} applies equally well 
to solutions in $L^p_{loc}([0,\infty),\xR^d)$ and to  continuous solutions,
but in the former case the equality is  understood for almost every $t$.
Note that \eqref{eq:16} ensures that the two sums in \eqref{eq:250} are 
equal, and also that they do not
depend on the values $\phi(y)$ for
$y\notin [-\tau_N,0]$ (which are not defined).


\textit{Proof of the Fact.}
For $0\leq t<\tau_1$, Equation \eqref{syst_delay_generique}
is of the form \eqref{eq:250}, with $M_q(t)=D_i(t)$ if $\sigma_q=\tau_i$
and $M_q(t)=0$ otherwise. If we assume inductively such a 
formula for $0\leq t<\sigma_{q_0}$ and substitute it in
the right hand side of \eqref{syst_delay_generique} 
when $\sigma_{q_0}\leq t<\sigma_{q_0+1}$ 
to express $z(t)$ as a linear combination of 
the $\phi(t-\sigma_q)$ for $-\tau_N\leq t-\sigma_q<0$, a moment's thinking 
will convince the reader that we get a formula of the same type
over the interval $\sigma_{q_0}\leq t<\sigma_{q_0+1}$ 
by defining  $M_q(t)$ as the sum of the coefficients corresponding,
after the above substitution,  to one and 
the  same $\phi(t-\sigma_q)$  (the latter may arise as many times as there are
decompositions $\sigma_q=\tau_i+\sigma_{q^\prime}$ with $i\in\{1,\cdots,N\}$
and $\sigma_{q^\prime}\in \Sigma_{\sigma_{q_0}+\tau_N}$. Such coefficients
are of the form $D_i(t)M_{q^\prime}(t)$, and therefore properties
(\textit{i}), (\textit{ii}) and (\textit{iii}) are obviously met.
This proves the fact.

Although we will not need this, it is 
instructive to derive an explicit expression for $M_q$
 that should be compared with \cite[thm. 3.14]{Chitour2016} or, in the continuous case, with  \cite[ch. 9, eqns. (1.4)-(1.5)]{Hale}. 
Namely, we can take  $M_q(t)$ to be the sum of all terms
\begin{equation}
\label{expfMq}
  \mathds{1}_{(\rho_{s-1}, \rho_{s}]}(t) \, D_{k_1}(t)D_{k_2}(t-\rho_{1})D_{k_3}(t-\rho_{2}) \cdots D_{k_s}\bigl(t-\rho_{s-1}\bigr)
\end{equation}
for all $s$ in $\xN\setminus\{0\}$ and all $s$-tuples $(k_1,\ldots,k_s)\in\{1,\cdots, N\}^s$ such that
$\sum_{j=1}^s \tau_{k_j}=\sigma_q$, where the numbers $\rho_j$ are defined by $\rho_0=0$ and 
$\rho_j=\sum_{i=1}^j\tau_{k_i}$ for  $j\geq1$ (in particular
$\rho_s=\sigma_q$), and
$\mathds{1}_{(\rho_{s-1}, \rho_{s}]}$ is the characteristic function 
of the interval $(\rho_{s-1}, \rho_{s}]$.
These maps $M_q$ satisfies \eqref{eq:16} because
$(\rho_{s-1},\rho_{s}]=(\sigma_q-\tau_{k_s},\sigma_q]$ is a subset of $(\sigma_q-\tau_N, \sigma_q]$,
and formula \eqref{eq:250} is easily checked from \eqref{syst_delay_generique},
by induction on $j$ such that $t\in(\sigma_{j-1},\sigma_j]$.

\begin{proof}[Proof of Theorem~\ref{th:L2impliesC0}]
Assume first that $1\leq p<\infty$.  
If System \eqref{syst_delay_generique} is $L^p$ exponentially stable,
there is by definition  
  $\gamma>0$ and $C_0>0$ such that, for all $\phi\in L^p([-\tau_N,0], \mathbb{R}^{d})$ and all
  $t>0$, one has
\begin{equation}\label{eq:253}
\Bigl(\int_{t-\tau_N}^{t}\|z(\uuuu)\|^pd\uuuu\Bigr)^{1/p} \! \leq C_0\,e^{-\gamma t}\|\phi\|_{L^p([-\tau_N,0],\xR^d)}
\end{equation}
for $z(.)$ the unique
solution of \eqref{syst_delay_generique} with initial condition $\phi$ given by
Theorem~\ref{existence_solution_delay}. Pick  
$t_{\star} \in (-\tau_N,0)$,
$v \in \mathbb{R}^{d}$, $\varepsilon>0$, and  define a function
$\phi_{t_{\star},v,\varepsilon}\in L^p([-\tau_N,0], \mathbb{R}^{d})$ by
\begin{equation}
\label{defcib}
\phi_{t_{\star},v,\varepsilon}(\theta)=\frac{1}{\varepsilon^{1/p}}
\mathds{1}_{(t_{\star}-\varepsilon,t_{\star})}(\theta) \,v\,,\quad\qquad\theta\in[-\tau_N,0].
\end{equation}
Let $z_{t_{\star},v,\varepsilon}(.)$ be the solution to
\eqref{syst_delay_generique} with initial condition $\phi_{t_{\star},v,\varepsilon}$ on
$[-\tau_N,0]$.
By \eqref{eq:253}, it holds  that
\begin{equation}
  \label{eq:256}
  \Bigl(\int_{t-\tau_N}^{t}\|z_{t_{\star},v,\varepsilon}(\uuuu)\|^pd\uuuu\Bigr)^{1/p} \leq C_0\,e^{-\gamma t}\|v\|\,,\qquad t>0,
\end{equation}
and from \eqref{eq:250} we get for all $\uuuu>0$ that
\begin{eqnarray}\label{eq:251}
  z_{t_{\star},v,\varepsilon}(\uuuu)&=&\frac{1}{\varepsilon^{1/p}}
                                        \left( \sum_{q=1}^{+\infty}  
                                      \mathds{1}_{(t_{\star}-\varepsilon,t_{\star})}(\uuuu-\sigma_q)
                                      \;M_q(\uuuu)\right) v.
\end{eqnarray}
Let us fix $t>0$ for a while.
By \eqref{eq:16}, the only terms  in the  sum on the right of
\eqref{eq:251} which  may not be zero for  a.e. $u\in (t-\tau_N,t)$
are such that
$E(q,t_{\star},\varepsilon):=(\sigma_q +t_{\star}-\varepsilon, \sigma_q +t_{\star}) \cap (\sigma_q-\tau_N, \sigma_q]
\cap(t-\tau_N,t)$ has strictly positive measure. The set of integers $q$ for which this holds  for some
$t_{\star}\in(-\tau_N,0)$ and some $\varepsilon>0$ consists exactly 
of those $q$ such that $t-\tau_N<\sigma_q<t+\tau_N$.
If we pick one of them and, say,  $\sigma_q\geq t$, it is easy to check that
for $\varepsilon$ small enough
$E(q,t_{\star},\varepsilon)=(t_{\star}+\sigma_q-\varepsilon,t_{\star}+\sigma_q)$
when $t_{\star}\leq t-\sigma_q$ and $E(q,t_{\star},\varepsilon)=\varnothing$ 
when $t_{\star}> t-\sigma_q$.
If on the contrary $\sigma_q<t$, then 
$E(q,t_{\star},\varepsilon)=(t_{\star}+\sigma_q-\varepsilon,t_{\star}+\sigma_q)$
when $t_{\star}> t-\sigma_q-\tau_N$ and $E(q,t_{\star},\varepsilon)=\varnothing$ 
when $t_{\star}\leq t-\sigma_q-\tau_N$. Altogether, since there are finitely many $q$
under examination ({\it i.e.} at most $Q(t+\tau_N)$), we can take 
$\varepsilon>0$  so small that all intervals $E(q,t_{\star},\varepsilon)$
are disjoint, and then we deduce from \eqref{eq:251} and the previous discussion that
\begin{eqnarray}
  \label{eq:254}
  \int_{t-\tau_N}^{t}\|z_{t_{\star},v,\varepsilon}(u)\|^pdu=
  \sum_{\{q:\,\,t-\tau_N-t_{\star}<\sigma_q\leq t-t_{\star}\}} 
  \frac{1}{\varepsilon} \int_{-\varepsilon}^{0} \|M_q(t_{\star}+\sigma_q+\theta)\,v\|^p\mathrm{d}\theta.
\end{eqnarray}
Observe next that a.e.  $t_{\star}\in(-\tau_N,0)$ is a Lebesgue point of (each entry of)
$s\mapsto M_q(s+\sigma_q)$ for all $q\in\xN$, and let $E$ denote the set of 
such points. By the triangle inequality, $E$ {\it a fortiori}
consists of Lebesgue points of  $s\mapsto \|M_q(s+\sigma_q)v\|$, 
and since $M_q\in L^\infty_{loc}(\xR,\xR^{d\times d})$ it 
also consists of Lebesgue points of
$s\mapsto \|M_q(s+\sigma_q)v\|^p$,  by the smoothness of
$x\to x^p$ for $x>0$. Thus, from \eqref{eq:256} and
\eqref{eq:254}, we deduce  on letting $\varepsilon\to0$ that
\begin{equation}
  \label{eq:into}
  t-\tau_N-t_{\star}<\sigma_q\leq t-t_{\star} \ \Rightarrow\ 
\|M_q(t_{\star}+\sigma_q)v\|  \leq C_0e^{-\gamma t}\|v\|,\qquad t_{\star}\in E.
\end{equation}
Now, choose $\sigma_q\in\Sigma$ and $t_{\star}\in E$. We can find $t>0$ such that
$t-\tau_N<\sigma_q+t_{\star}<t$ and then, applying what precedes with this $t$ 
and this $t_{\star}$, we obtain in view of \eqref{eq:into} that
\begin{equation}
\label{estfq}
\|M_q(t_{\star}+\sigma_q)v\|  \leq C_0e^{-\gamma t}\|v\|\leq
C_0e^{-\gamma (t_{\star}+\sigma_q)}\|v\|.
\end{equation}
As $E$ has full measure in $(-\tau_N,0]$ and $v\in\xR^d$ is arbitrary,
 we conclude from \eqref{estfq}
and \eqref{eq:16} that
\begin{eqnarray}
\label{eq:255.5}
\vertiii{M_q(s)}  \leq C_0\,e^{-\gamma s },\qquad \text{a.e. } s>0,
\end{eqnarray}
where $\vertiii{\cdot}$ is the spectral norm for matrices 
on Euclidean space.
Because the number of summands  in the middle term
of \eqref{eq:250} is $Q(t+\tau_N)$ which is bounded above by
$C t^N$ for some constant $C$, as asserted in \eqref{eq:249a},
the inequality \eqref{eq:255.5} implies that to any 
$\gamma^\prime\in(0,\gamma)$ there is a constant $C_1>0$ for which
\begin{align}\label{eq:finpfin}
&\|z\|_{L^\lambda((t-\tau_N,t),\xR^d)}\leq C_1\,e^{-\gamma^\prime
  t}\|\phi\|_{L^\lambda((-\tau_N,0),\xR^d)},\  t\geq 0,\  1\leq \lambda<\infty,
\\
\label{casinfpfin}
\text{and also}\ 
&\|z\|_{L^\infty((t-\tau_N,t),\xR^d)} \leq C_1 e^{-\gamma^\prime t} 
\|\phi\|_{L^\infty((-\tau_N,0),\xR^d)} , \  t\geq 0.
\end{align}
Since \eqref{casinfpfin} readily implies $C^0$ exponential stability when the maps $D_i(.)$ are continuous, this achieves the proof when $1\leq p<\infty$.

Assume now that $p=\infty$, so that \eqref{eq:253} gets replaced by
$\|z\|_{L^\infty((t-\tau_N,t),\xR^d)} \leq C_0 e^{-\gamma t} 
\|\phi\|_{L^\infty((-\tau_N,0),\xR^d)}$ for all $t>0$.
The goal is again to prove \eqref{eq:255.5} from which the result follows,
as we just saw. For this, we argue much like we did before,
defining $\phi_{t_{\star},v,\varepsilon}$ as in \eqref{defcib}
except that we do not divide by $\varepsilon^{1/p}$. Then,
\eqref{eq:256} becomes 
\begin{equation}
\label{expmajinf}
\underset{\alpha \in (t-\tau_N,t)}{\mbox{\rm ess.\,}\sup} \|z_{t_{\star},v,\varepsilon}(\alpha)\|  \leq C_0e^{-\gamma t}\|v\|,
\end{equation}
 and the discussion that led us
to \eqref{eq:254} now yields  for $\varepsilon>0$ small enough:
\begin{eqnarray}
  \label{eq:254inf}
  \underset{\alpha \in (t-\tau_N,t)}{\mbox{\rm ess.\,}\sup} \|z_{t_{\star},v,\varepsilon}(\alpha)\| =  \max_{\{q:\,\,t-\tau_N-t_{\star}<\sigma_q\leq t-t_{\star}\}} 
 \ \ \underset{\theta \in (-\varepsilon,0)}{\mbox{\rm ess.\,}\sup} \|M_q(t_{\star}+\sigma_q+\theta)\,v\|.
\end{eqnarray}
We need now to replace Lebesgue points by points of approximate continuity. 
Recall that a function $f:\xR\to\xR^m$ is approximately continuous at
$x$ if, for every $\epsilon>0$,
\[
\lim_{r\to0^+}\frac{1 }{2\,r}\,\mathcal{H}^1\Bigl((x-r,x+r)\cap\{y:\|f(y)-f(x)\|>\epsilon\}\Bigl)=0
\]
where $\mathcal{H}^1$ is the 
Lebesgue measure on $\xR$, and that a measurable $f$ is 
approximately continuous at almost every point
\cite[thm. 1.37]{EvaGar1992}. Thus, if we define
$E\subset(-\tau_N,0)$ to be the set of approximate continuity points
of all maps $s\mapsto M_q(s+\sigma_q)$ (say, extended by zero off
$(-\tau_N,0)$) as  $q$ ranges over $\xN$, then  $E$ has full measure
in $(-\tau_N,0)$ and letting $\varepsilon\to0$ in \eqref{eq:254inf} we find 
that \eqref{eq:into} holds. From the latter we obtain \eqref{eq:255.5} 
by the same reasoning as before, thereby completing the proof when $p=\infty$.

It remains to handle the case where the $D_i(t)$ are continuous and
System \eqref{syst_delay_generique} is $C^0$ exponentially stable.
Then, 
the previous argument needs adjustment because
$\phi_{t_{\star},v,\varepsilon}\notin \mathcal{C}$. However, it is easy to 
construct a sequence of continuous functions
$\varphi_k:[-\tau_N,0]\to[0,1]$, with $\varphi_k(0)=\varphi_k(-\tau_j)=0$ for
$1\leq j\leq N$, such that $\varphi_k$ converges pointwise a.e. to 
$\mathds{1}_{(t_{\star}-\varepsilon,t_{\star})}$ when $k\to+\infty$
(for instance, we may take 
piecewise linear $\varphi_k$). Then, the 
$\phi_k(\theta):=\varphi_k(\theta)v$ lie in $\mathcal{C}$, 
and if $z_k$ denotes the solution to System \eqref{syst_delay_generique} 
with initial condition $\phi_k$,  we get by assumption 
that $\underset{\alpha \in (t-\tau_N,t)}{\sup} \|z_k(\alpha)\|  \leq C_0e^{-\gamma t}\|v\|$. As $\phi_k$ converges pointwise a.e. 
to $\phi_{t_\star,v,\varepsilon}$ on $[-\tau_N,0]$,
we see from \eqref{eq:250} that
$z_k$ converges to $z_{t_\star,v,\varepsilon}$ pointwise a.e. on $\xR$.
Thus, letting $k\to+\infty$, we deduce that \eqref{expmajinf} holds
and we conclude as before.
\end{proof}

\subsection{Proof of Theorem~\ref{theorem-central-delay}}
\label{sec:proof-thDelay}

First assume that each set $\mathcal{I}_j$ has even cardinality $2\,n_j$,
and put $N=\sum_{j=1}^{\widehat{M}}n_j$ so that $d=2N$.
Let $P_3$ be the permutation matrix sending
$\mathcal{I}_1$ to $\{1,\ldots,n_1\}\cup\{N+1,\ldots,N+n_1\}$
and, more generally,  
$\mathcal{I}_j$ to
$\{1+\sum_{\ell=1}^{j-1}n_\ell,\ldots,\sum_{\ell=1}^{j}n_\ell\} \cup
\{N+1+\sum_{\ell=1}^{j-1}n_\ell,\ldots, N+\sum_{\ell=1}^{j}n_\ell\}$ for each 
$j$.
Set $\tau_k=\eta_j$ for each $k$ in $\{1+\sum_{\ell=1}^{j-1}n_\ell,\ldots,
\sum_{\ell=1}^{j}n_\ell\}$.
Using $P_3$ as change of basis and denoting by $(x_1,\ldots,x_N,y_1,\ldots,y_N)$ the new coordinates, one can (by gathering the matrices with disjoint 
nonzero columns into a single one) re-write
\eqref{eq:25} as \eqref{eq:10} where $-\left(I+\mathbf{A}(t)\,\mathbf{K}\right)^{-1}
\left(I-\mathbf{A}(t)\,\mathbf{K}\right)P_2$ has been replaced with 
$\sum_{j=1}^{\widehat{\newN}}{P_3}^{-1}\, \widehat{D}_j(t)\,P_3$.
We want now to find  $\mathbf{A}(t)$ and $\mathbf{K}$ so that these 
two matrices coincide. For this, we fix
$\mathbf{K}=Id$ and solve
$\left(I+\mathbf{A}(t)\right)^{-1}\left(I-\mathbf{A}(t)\right)=R(t)$
with respect to $\mathbf{A}(t)$, where
$R(t) =-{P_3}^{-1}\left(\sum_{j=1}^{\widehat{\newN}}\widehat{D}_j(t)\right)P_3{P_2}^{-1}
=-{P_3}^{-1}\left(\sum_{i=1}^{\newN}D_i(t)\right)P_3{P_2}^{-1}$ (the last equality is clear from the
definition of $\widehat{D}_i(t)$ in \eqref{eq:21}).
Assumption \textup{(\textit{ii})} implies $\vertiii{R(t)}\leq\nu<1$ because $P_2$ and $P_3$ are orthogonal
matrices, hence, according to Lemma~\ref{lemma1}, setting $\mathbf{A}(t)=(Id-R(t))(Id+R(t))^{-1}$
solves the above and satisfies Assumption~\ref{ass:dissip} with $\alpha=(1-\nu)/(1+\nu)$;
\eqref{eq:9.5} is satisfied too with $\mathbf{K}=Id$, setting all the numbers $K_k$ to 1.
By virtue of Proposition~\ref{prop:equiv} and Theorem~\ref{theorem-central}, the difference-delay
equation \eqref{eq:10} with these $\mathbf{A}(.)$ and $\tau_k$ is $L^p$ exponentially stable for all
$p\in[1,\infty]$, as well as $C^0$ exponentially stable if the maps $D_i(.)$ (hence 
$\mathbf{A}(.)$) are continuous.
This proves the result if all the sets $\mathcal{I}_j$ has even cardinality.

If some of the sets $\mathcal{I}_j$ have odd cardinality,
define $d'>d$ so that $d'-d$ is the number of such sets $\mathcal{I}_j$.
By adjoining to each such $\mathcal{I}_j$ one element of $\{d+1,\ldots,d'\}$, one constructs
a partition $\widetilde{\mathcal{I}}_1,\ldots,\widetilde{\mathcal{I}}_{\widehat{\newN}}$ of $\{1,\ldots,d'\}$
such that,  for each $j$, $\widetilde{\mathcal{I}}_j$ has even cardinality and contains $\mathcal{I}_j$.
Constructing some $d'\times d'$ matrices $\widetilde{D}_j(t)$ by adding $d'-d$ zero last lines and
$d'-d$ zero last columns to $\widehat{D}_j(t)$, the  following difference-delay
system (with state $\widetilde{z}$ in $\xR^{d'}$):
  \begin{equation}
    \label{eq:25}
    \widetilde{z}(t)=\sum_{j=1}^{\widehat{\newN}}\widetilde{D}_j(t)\widetilde{z}(t-\eta_j)
  \end{equation}
satisfies the assumptions of the theorem: (\textit{i}) with the sets $\widetilde{\mathcal{I}}_j$ instead of
the original sets $\mathcal{I}_j$ and (\textit{ii}) because adding zero lines and columns to a matrix does not
increase its norm, hence the first part of the proof gives exponential stability, that yields
exponential stability of the original system because, since the last $d'-d$ columns are zero, the
evolution of the $d$ first entries of $z$ does not depend on the last ones.
\hfill$\square$

\clearpage

\section*{Acknowledgements}
The authors are indebted to Prof.\ J.-M.\ Coron (Sorbonne Universit\'e, Paris) for
interesting discussions on the reference \cite{CoNg}.

\bibliographystyle{siamplain}
\bibliography{Biblio_these}
\end{document}